\documentclass[reqno]{amsart}

\usepackage{amsmath,amssymb,amsthm,amsfonts}
\usepackage{hyperref}

\usepackage{color}

\newcommand{\Blue}{\textcolor{blue}}

\newtheorem{lemma}{Lemma}[section]
\newtheorem{theorem}{Theorem}[section]

\newtheorem{remark}{Remark}[section]

\numberwithin{equation}{section}

\newcommand{\dis}{\displaystyle}

\newcommand{\R}{\mathbb{R}}

\newcommand{\semiG}{\mathbb{A}}
\newcommand{\highG}{\Theta}
\newcommand{\highB}{\Lambda}

\newcommand{\sourceG}{S}

\newcommand{\FD}{\mathbf{D}}

\newcommand{\CD}{\mathcal{D}}
\newcommand{\CE}{\mathcal{E}}

\newcommand{\CI}{\mathcal{I}}
\newcommand{\CJ}{\mathcal{J}}

\newcommand{\CN}{\mathcal{N}}

\newcommand{\na}{\nabla}

\newcommand{\al}{\alpha}
\newcommand{\be}{\beta}
\newcommand{\ga}{\gamma}

\newcommand{\la}{\lambda}
\newcommand{\de}{\delta}
\newcommand{\si}{\sigma}
\newcommand{\pa}{\partial}
\newcommand{\ka}{\kappa}
\newcommand{\eps}{\epsilon}

\newcommand{\vth}{\vartheta}

\newcommand{\Ga}{\Gamma}

\newcommand{\lag}{\langle}
\newcommand{\rag}{\rangle}

\begin{document}

\title[Global smooth dynamics of a fully ionized plasma]{Global smooth dynamics of a fully ionized plasma with long-range collisions}

\author[Renjun Duan]{Renjun Duan}
\address{Department of Mathematics, The Chinese University of Hong Kong,
Shatin, Hong Kong}
\email{rjduan@math.cuhk.edu.hk}
\date{\today}

\begin{abstract}
The motion of a fully ionized plasma of electrons and ions is generally governed by the Vlasov-Maxwell-Landau system. We prove the global existence of solutions near Maxwellians to the Cauchy problem of the system for the long-range collision kernel of soft potentials, particularly including the classical Coulomb collision, provided that initial data is smooth enough and decays in velocity variable fast enough. As a byproduct, the convergence rates of solutions are also obtained. The proof is based on the energy method through designing a new temporal energy norm to capture different features of this complex system such as dispersion of the macro component in $\R^3$, singularity of the long-range collisions and regularity-loss of the electromagnetic field.  
\end{abstract}


\maketitle
\thispagestyle{empty}


\section{Introduction}

\subsection{Kinetic equations for a fully ionized plasma}

The motion of a fully ionized plasma consisting of only two species particles (electrons and ions) under the influence of the self-consistent Lorentz force and binary collisions is governed by the kinetic transport equations
\begin{equation}
\label{V}
\begin{split}
 & \pa_t F_+ +\xi\cdot \na_x F_+ + (E+\xi \times B) \cdot \na_\xi F_+ =
  Q(F_+,F_+)+Q(F_+,F_-),
  \\
 & \pa_t F_- +\xi\cdot \na_x F_- - (E+\xi \times B) \cdot \na_\xi F_- =
  Q(F_-,F_+)+Q(F_-,F_-),
  \end{split}
\end{equation}
coupled with the Maxwell equations
\begin{equation}
\label{M}
\begin{split}
&\pa_t E-\na_x\times B=-\int_{\R^3} \xi (F_+-F_-)\,d\xi,\\
&\pa_t B+\na_x\times E =0,\\
&\na_x\cdot E =\int_{\R^3} (F_+-F_-)\,d\xi,\ \ \na_x \cdot
B=0.
\end{split}
\end{equation}
Here $F_\pm=F_\pm(t,x,\xi)\geq 0$ stands for the number densities of ions
$(+)$ and electrons $(-)$ which have position
$x=(x_1,x_2,x_3)\in\R^3$ and velocity
$\xi=(\xi_1,\xi_2,\xi_3)\in\R^3$ at time $t\geq 0$, and $E(t,x)$, $B(t,x)$ denote the
electro and magnetic fields, respectively.
The initial data of the system of equations \eqref{V}-\eqref{M} is given by
\begin{eqnarray*}
&\dis F_\pm(0,x,\xi)=F_{0,\pm}(x,\xi),\ \ E(0,x)=E_0(x),\ \
B(0,x)=B_0(x),
\end{eqnarray*}
satisfying the compatibility conditions
\begin{equation*}
\na_x\cdot E_0=\int_{\R^3}(F_{0,+}-F_{0,-})\,d\xi,\ \ \na_x\cdot
B_0=0.
\end{equation*}

Notice that for simplicity of presentation and without loss of generality, all the physical parameters appearing in the system, such as the particle masses and the light speed, and other involving constants, have been chosen to be unit. Moreover, the relativistic effects are neglectd in our discussion of the model. This is normally justified if the plasma temperature is much lower than the electron rest mass, cf.~\cite[Section 3.6]{HS}. We refer interested readers to \cite[Chapters 2 and 3]{HS} or \cite[Chapter 6]{KT} for the non dimensional representation of the system \eqref{V}-\eqref{M} and also its physical importance in the study of plasma transport phenomena.

The quadratically nonlinear operator $Q$ in \eqref{V}  is described by  the Landau collision mechanism between particles in the form of 
\begin{eqnarray*}
&\dis Q(F,G)=\na_\xi\cdot
\left\{\int_{\R^3}\phi(\xi-\xi_\ast)[{\na_{\xi}F(\xi)
G(\xi_\ast)-F(\xi) \na_{\xi_\ast}G(\xi_\ast)}]\,d\xi_\ast\right\}.
\end{eqnarray*}
The non-negative matrix $\phi$ denotes the Landau collision kernel
\begin{equation*}
    \phi^{ij}(\xi)=C_\phi|\xi|^{\ga+2}\left\{\de_{ij}-\frac{\xi_i\xi_j}{|\xi|^2}\right\},\quad -3\leq \ga< -2,
\end{equation*}
for a constant $C_\phi>0$. Note that $\ga=-3$ corresponds to the Coulomb potential for the classical Landau operator which is originally found by Landau (1936). For the mathematical discussions of the general collision kernel with $\ga>-3$, see the review paper \cite{Villani1} by Villani. 

From now on, for simplicity of presentation,  it is convenient to call \eqref{V}-\eqref{M} the Vlasov-Maxwell-Landau system. In this paper we aim at proving  the global existence of solutions to the Cauchy problem of the system near a global Maxwellian under some conditions on initial data.

\subsection{Reformulation}
Write the normalized global Maxwellian as
\begin{equation}
\notag
\mu=\mu(\xi)=(2\pi)^{ -3/2}e^{-|\xi|^2/2},
\end{equation}
and set
the perturbation in the standard way
\begin{equation}
\notag
F_\pm(t,x,\xi)=\mu + \mu^{1/2} f_\pm(t,x,\xi).
\end{equation}
Use $[\cdot,\cdot]$ to denote the column vector. Set
$F=[F_+,F_-]$ and $
f=[f_+,f_-]$.
Then the Cauchy
problem \eqref{V}-\eqref{M} can be
reformulated as
\begin{equation}\label{VM}
\left\{\begin{split}
&\pa_t f+\xi\cdot\na_x f + q_0 (E+\xi \times B)\cdot \na_\xi f-
E\cdot \xi \mu^{1/2} q_1 +Lf
\\
&\dis \hspace{6cm}=\frac{q_0}{2}E\cdot \xi f +\Ga(f,f),
\\
&\dis \pa_t E-\na_x\times B=-\int_{\R^3} \xi \mu^{1/2}(f_+-f_-)\,d\xi,\\
&\dis \pa_t B+\na_x\times E =0,\\
&\dis\na_x\cdot E =\int_{\R^3}\mu^{1/2} (f_+-f_-)\,d\xi,\ \ \na_x \cdot
B=0,
\end{split}\right.
\end{equation}
with initial data
\begin{eqnarray}\label{VM0}
&\dis f_\pm(0,x,\xi)=f_{0,\pm}(x,\xi),\ \ E(0,x)=E_0(x),\ \
B(0,x)=B_0(x),
\end{eqnarray}
satisfying the compatibility condition
\begin{equation}\label{VM1}
\na_x\cdot E_0=\int_{\R^3} {\mu^{1/2}}(f_{0,+}-f_{0,-})\,d\xi,\ \ \na_x\cdot
B_0=0.
\end{equation}
Here,  $q_0={\rm diag} (1,-1)$, $q_1=[1,-1]$, and the linearized
collision term $L f$ and the nonlinear collision term $\Ga (f,f)$ are
respectively defined by
\begin{equation}
\notag
    L f = [L_+ f, L_- f],\ \ \Ga(f,g)=[\Ga_+(f,g),\Ga_-(f,g)],
\end{equation}
with
\begin{eqnarray*}
 L_{\pm} f &=&{-}2\mu^{-1/2} Q(\mu^{1/2} f_\pm, \mu)\Blue{-}\mu^{-1/2}
    Q(\mu,\mu^{1/2}\{f_\pm+f_\mp\}),\\
\Ga_\pm (f,g)&=&\mu^{-1/2} Q(\mu^{1/2} f_\pm, \mu^{1/2}
    g_\pm)+\mu^{-1/2} Q(\mu^{1/2} f_{{\pm}}, \mu^{1/2} g_{{\mp}}).
\end{eqnarray*}

\subsection{Macro projection, weights and norms}  
As in \cite{Guo-VPL, Guo3}, the null space of the linearized operator $L$ is given by
\begin{equation}
\notag
    \CN={\rm span}\left\{[1,0]\mu^{1/2},\ [0,1]\mu^{1/2}, [\xi_i,\xi_i]\mu^{1/2}
 ~ (1\leq i\leq 3),
    [|\xi|^2,|\xi|^2]\mu^{1/2} \right\}.
\end{equation}
Let $P$ be the orthogonal projection from $L^2_\xi\times L^2_\xi$ to $\CN$.
Given $f(t,x,\xi)$, one can write $P$ as
\begin{multline}
\notag
 P f=a_+(t,x)[1,0]\mu^{1/2}+a_-(t,x)[0,1]\mu^{1/2}
    \\
    +\sum_{i=1}^3b_i(t,x)[1,1]\xi_i\mu^{1/2}+c(t,x)[1,1](|\xi|^2-3)\mu^{1/2},
\end{multline}
where the coefficient functions are determined by $f$ in the way that
\begin{equation}
\notag
\begin{split}
  \dis   & a_\pm= \langle \mu^{1/2}, f_\pm\rangle= \langle \mu^{1/2}, P_\pm f\rangle,
  \\
  \dis  &b_i=\frac{1}{2}\langle \xi_i \mu^{1/2}, f_++f_-\rangle
=\langle \xi_i \mu^{1/2},P_\pm f\rangle,
\\
  \dis  &c= \frac{1}{12}\langle (|\xi|^2-3) \mu^{1/2}, f_+ + f_-\rangle
= \frac{1}{6}\langle (|\xi|^2-3) \mu^{1/2}, P_\pm f\rangle.
    \end{split}
\end{equation}

In what follows, we introduce the weight functions and norms used for the presentation of the main result later on. Fix a constant $\vth$ with $0<\vth\leq 1/4$.
Define
\begin{equation}
\notag
    w_{\tau,\la}=w_{\tau,\la}(t,\xi)=\lag \xi \rag^{(\ga+2) \tau} \exp\left\{ \frac{\la}{(1+t)^\vth}\lag \xi\rag^2\right\},
\end{equation}
where constants $\tau\in \R$ and $\la\geq 0$ are two parameters which may vary in different places. For $f=f(t,x,\xi)$,
define
\begin{equation}
\notag
|f(x)|_{\tau,\la}^2=\int_{\R^3}w_{\tau,\la}^2(t,\xi) |f|^2\,d\xi,\quad \|f\|_{\tau,\la}^2=\int_{\R^3} |f(x)|_{\tau,\la}^2\,dx,
\end{equation}
and
\begin{eqnarray*}
&\dis   |f(x)|_{\FD,\tau,\la}^2=\sum_{i,j=1}^3\int_{\R^3}w_{\tau,\la}^2(t,\xi)\left\{\si^{ij}
    \pa_if\pa_j f+\si^{ij}\frac{\xi_i}{2}\frac{\xi_j}{2}|f|^2\right\}\,d\xi,\\
    &\dis  \|f\|_{\FD,\tau,\la}^2=\int_{\R^3}|f(x)|_{\FD,\tau,\la}^2\,dx,\notag
\end{eqnarray*}
where $\pa_i=\pa_{\xi_i}$ denotes the velocity derivative with respect to $\xi_i$, and $\si^{ij}=\si^{ij}(\xi)$ is the Landau collision frequency given by
\begin{equation}
\notag
    \si^{ij}(\xi)=\phi^{ij}\ast \mu (\xi)=\int_{\R^3} \phi^{ij}(\xi-\xi_\ast)\mu(\xi_\ast)\,d\xi_\ast.
\end{equation}

To study the global existence through the energy method, the temporal energy functional and the corresponding dissipation rate are defined by
\begin{eqnarray}
\CE_{N,\ell,\la}(t)&\sim& \sum_{|\al|+|\be|\leq N}\|\pa_\be^\al f(t)\|_{|\be|-\ell,\la}^2 +\|(E,B)\|_{H^N}^2,
\label{def.e}
\end{eqnarray}
and
\begin{eqnarray}
  \CD_{N,\ell,\la}(t)&=& \sum_{|\al|+|\be|\leq N}\|\pa_\be^\al \{I-P\}f(t)\|_{\FD,|\be|-\ell,\la}^2\notag\\
  &&+\sum_{|\al|\leq N-1} \|\na_x \pa^\al (a_\pm, b,c)\|^2\notag\\
  &&+ \|a_+-a_-\|^2+ \|E\|_{H^{N-1}}^2 +\|\na_x B\|_{H^{N-2}}^2\notag\\
  &&+\frac{\la}{(1+t)^{1+\vth}} \sum_{|\al|+|\be|\leq N}\|\lag \xi\rag \pa_\be^\al \{I-P\}f(t)\|_{|\be|-\ell,\la}^2,\label{def.dr}
\end{eqnarray}
where the integer $N\geq 0$ and $\ell\geq 0$ are parameters which may differ in different places.  For simplicity, we write $w_{\tau}=w_{\tau,0}$ for $\la=0$, and $\CE_{N}(t)=\CE_{N,0,0}(t)$ for $\ell=\la=0$, and likewise for others. Moreover, regarding $\CE_{N,\ell,\la}(t)$ and $\CD_{N,\ell,\la}(t)$, whenever the case $\la>0$ occurs, we always require $\ell-N\geq 0$; it means that the algebraic weight factor is always of positive power when the exponential weight factor is present. Notice that the last term of $ \CD_{N,\ell,\la}(t)$ in \eqref{def.dr} disappears when $\la=0$.

Let constants $N_0$ and $\ell_0$ be fixed properly large, and let constants $\la_0>0$ and $\eps_0>0$ be fixed properly small; the choice of $N_0$, $\ell_0$, $\la_0$ and $\eps_0$ can be seen in the late proof. Set  $N_1=\frac{3}{2}N_0$ and $\ell_1=\frac{1}{2}\ell_0$. 
The temporal energy norm $X(t)$ is defined by
\begin{eqnarray}
  X(t) &=&  \sup_{0\leq s\leq t} \{\CE_{N_1}(s)+(1+s)^{\frac{3}{2}}\CE_{N_1-2}(s)\}\notag\\
  && +\sup_{0\leq s\leq t} \{(1+s)^{-\frac{1+\eps_0}{2}}\CE_{N_1,\ell_1,\la_0}(s)+\CE_{N_1-1,\ell_1,\la_0}(s)\notag\\
  &&\qquad\qquad\qquad\qquad\qquad\qquad+ (1+s)^{\frac{3}{2}} \CE_{N_1-3,\ell_1-1,\la_0}(s)\}\notag\\
  &&+\sup_{0\leq s\leq t} \{ \CE_{N_0,\ell_0,\la_0}(s)+(1+s)^{\frac{3}{2}}\CE_{N_0,\ell_0-1,\la_0}(s)\}\notag\\
  &&+ \sup_{0\leq s\leq t} \{(1+s)^{2(1+\vth)} \|\na_x (E,B)(s)\|_{H^{N_0-1}}^2\}.
  \label{def.X}
\end{eqnarray}
Notice that in order for the algebraic weight factor to gain the  large enough positive power when there is the exponential weight factor $\exp\{\la_0\lag\xi\rag^2/(1+t)^{\vth}\}$, we also let $\ell_0-3N_0$ be properly large.

\subsection{Main result} The main result of the paper is stated as follows.

\begin{theorem}\label{thm.gl}
Assume $-3\leq \ga<-2$. Take $0<\vth\leq 1/4$, and also take constants $N_0$, $\ell_0$ properly large with $\ell_0-3N_0$ properly large, and constants $\la_0>0$, $\eps_0>0$ properly small. Fix a constant $\ell_2>\frac{5}{4}+\frac{N_0}{2}$.
Let $f_0=[f_{0,+},f_{0,-}]$ satisfy
$F_{\pm}(0,x,\xi)=\mu(\xi)+\mu^{1/2}(\xi)f_{0,{\pm}}(x,\xi)\geq 0$. If
\begin{multline}
\label{def.Y0}
Y_0=\sum_{|\al|+|\be|\leq N_0}\|\pa_\be^\al f_0\|_{|\be|-\ell_0,\la_0}+\sum_{|\al|+|\be|\leq N_1}\|\pa_\be^\al f_0\|_{|\be|-\ell_1,\la_0}\\
+\|(E_0,B_0)\|_{H^{N_1}\cap Z_1}
+\|w_{-\ell_2}f_0\|_{Z_1}
\end{multline}
is sufficiently small, then there are appropriately defined energy functionals $\CE_{N,\ell,\la}(t)$ appearing in $X(t)$ such that the Cauchy problem \eqref{VM}, \eqref{VM0}, \eqref{VM1} of the Vlasov-Maxwell-Landau system admits a unique global solution $\big(f(t,x,\xi),E(t,x),B(t,x)\big)$ satisfying
$F_{\pm}(t,x,\xi)=\mu(\xi)+\mu^{1/2}(\xi)f_{\pm}(t,x,\xi)\geq 0$ and
\begin{equation}
\label{thm.gl.1}
X(t)\lesssim Y_0^2,
\end{equation}
for all time $t\geq 0$.,
\end{theorem}

\begin{remark}
The following are several points  to remark on this theorem:
\begin{itemize}

\item[(a)] 
Although only the soft potential case $-3\leq \ga<-2$ is considered here, the hard case for $\ga\geq -2$ could be much simpler.
  
\item[(b)] 
Even when $\vth$ is strictly less than $1/4$, it can sitll be shown from the late proof that
  \begin{equation}
\notag
\sum_{1\leq |\al|\leq N_0}(\|\pa^\al f\|+\|\pa^\al (E,B)\|)
\end{equation} 
decays in time with the rate $(1+t)^{-5/4}$, cf.~\eqref{lem.emd.p3}.

\item[(c)] 
By the definition of $X(t)$, the uniform-in-time inequality \eqref{thm.gl.1} implies that the weighted high-order energy functional $\CE_{N_1,\ell_1,\la_0}(t)$, particularly 
\begin{equation}
\notag
\sum_{|\al|+|\be|=N_1} \|\pa_\be^\al \{I-P\}f\|_{|\be|-\ell_1,\la_0}^2,
\end{equation}
may increase in time with the rate $(1+t)^{(1+\eps_0)/2}$.

\item[(d)] 
Setting $N_1=\frac{3}{2}N_0$ and $\ell_1=\frac{1}{2}\ell_0$ is just for simplicity of presentation. The general choice of $N_1$, $\ell_1$ in terms of $N_0$, $\ell_0$ is possible. In addition, for brevity, $N_0$, $\ell_0$ and $\ell_0-3N_0$ are assumed to be properly large. We would not track in this paper the critical values of $N_0$, $\ell_0$ and all other parameters. However, it should remain an interesting problem to design a new energy norm with the optimal choice of regularity and velocity integrability on initial data in order to ensure the global existence of the Cauchy problem. 

\item[(e)] 
The similar approach developed in this paper could be immediately applied to the case of the periodic box. In that case, we need to assume all the conservation laws of the system so that the Poincar\'{e} inequality can be applied to deal with the zero-order dissipation of the macro component $Pf$ and the magnetic field $B$.
\end{itemize}
\end{remark}

In what follows we mention some work only related to this paper; interested readers may refer to them for more references therein. Notice that there
are different approaches in establishing  the mathematical theories on
 the Landau equation, see \cite{AV,DL, DV, Lio,Vi, Z}. In the perturbation framework, Guo \cite{Guo-L} firstly established the global existence for the purely Landau equation with Coulomb potentials in the absence of any force; see also \cite{HY}. Very recently the same author \cite{Guo-VPL} made a further progress for the Vlasov-Poisson-Landau system on the periodic box when the self-consistent potential force is present; see \cite{DYZ-VPL} and \cite{SZ} for generalizations of the result to the case of the whole space. We pointed out that Duan-Yang-Zhao \cite{DYZ-VPL} used a different approach arising from the study of the Vlasov-Poisson-Boltzmann system \cite{DYZ,DYZ-s}.

When the plasma transport is effected by the Lorentz force coupled with the Maxwell equations, there are two cases in which the global existence of solutions near the global Maxwellian have been well studied. One case is to take the Landau operator with $\ga\geq -1$. In this case the problem was solved by Guo \cite{Guo3} although the Vlasov-Maxwell-Boltzmann system for only the hard-sphere model is considered there.  Notice from Lemma \ref{lem.Ld} that  if $\ga\geq -1$, i.e.~$\ga+2\geq 1$, then compared to the soft potential case, the linearized Landau operator has the stronger dissipative property which can control those nonlinear terms with the velocity-growth rate $|\xi|$, typically occurring to $E\cdot \xi f$. The other case is the relativistic version of the Vlasov-Maxwell-Landau system studied by Strain-Guo \cite{SG-R}. In this case, due to the boundedness of the relativistic velocity and the special form of the relativistic Maxwellian, the velocity growth phenomenon in the nonlinear term disappears, and instead the more complex property of the relativistic collision operator was analyzed there.  

Though there are a few results mentioned above, it still remains unknown to obtain the global existence for the Vlasov-Maxwell-Landau system for the Coulomb potential in the classical sense because of the quite complex property of the system that we will point out in more detail next subsection. To the knowledge of our best, Theorem \ref{thm.gl} is the first result in this direction. The proof of  Theorem \ref{thm.gl} is based on the energy method, 
cf.~\cite{Guo-L, Guo3} or \cite{LY-S},  together with a new quite delicate bootstrap argument. In particular, the introduction of the temporal energy norm $X(t)$, which is the main strategy of the proof,  captures most of important features of the coupling system.

\subsection{Difficulty and idea in the proof}

The Vlasov-Maxwell-Landau system under consideration has the following three typical features which result to different mathematical difficulties:
\begin{itemize}
  \item the degeneration of dissipation at large velocity for the linearized Landau operator with soft potentials;
  \item the velocity-growth of the nonlinear term, as mentioned before; 
  \item the regularity-loss of the electromagnetic field.
\end{itemize} 
The first feature makes it impossible to control the velocity derivative of the linear transport term $\xi\cdot \na_x f$ without any velocity weight. To deal with it, the algebraic weight factor $\lag \xi\rag^{(\ga+2)|\be|}$ depending on the order of velocity differentiation was introduced in \cite{Guo-L}. This, however, induces another difficulty when estimating the nonlinear term $(E+\xi\times B)\cdot \na_\xi f$, since the term contains one velocity differentiation so that the extra velocity-growth with rate $|\xi|^{-(\ga+2)}$ is produced. Notice that this new trouble as well as the obvious velocity-growth in $E\cdot \xi f$ happen to the nonlinear term only. \cite{DYZ} then introduced a time-velocity dependent exponential weight factor $\exp\{\la_0\lag\xi\rag^2/(1+t)^{\vth}\}$ which from the weighted estimate on $\pa_t f$ indeed leads to the following additional good term in the dissipation rate
\begin{equation}
\notag
\frac{\la_0\vth}{(1+t)^{1+\vth}} \iint \lag \xi\rag^2 w_{|\be|-\ell,\la_0}^2(t,\xi)|\pa_\be^\al\{I-P\}f|^2\,dxd\xi.
\end{equation}
See also \cite{DYZ-s, DYZ-VPL}. Thus, as long as those nonlinear velocity-growth terms  contains a portion decaying in time faster than $(1+t)^{1+\vth}$, they can be controlled by the above dissipation term. 

The third feature mentioned before can be seen from the definition of the dissipation rate functional $\CD_{N,\ell,\la}(t)$ in which the $L^2$ norm of $N$th-order spatial derivatives of $(E,B)$ is missing. Here, the regularity-loss results essentially from the coupling of the hyperbolic Maxwell equations but not from the technique of the approach; see \cite{D-em} for the analysis of the Green's function of the damping Euler-Maxwell system. Due to the regularity-loss, two more difficulties appear. One difficulty is that one can not expect derivatives of the electromagnetic field  $(E,B)$ of all orders up to the largest number $N_1$ to decay time fast enough, so that the estimate on those nonlinear velocity-growth terms is still a problem. To solve it, when estimating the weighted inner product term
\begin{equation}
\notag
\sum_{|\al|+|\be|\leq N_1}\lag \pa^\al_\be \{\frac{1}{2}E\cdot \xi f-(E+\xi\times B)\cdot \na_\xi f\}, w_{|\be|-\ell_1,\la_0}^2\pa_\be^\al f\rag,
\end{equation}
we use the time-decay property for the only low-order derivatives of both $(E,B)$ and $f$. Notice that compared to the high-order energy functional $\CE_{N_1,\ell_1,\la_0}(t)$ of $f$, the low-order  one must be assigned with the higher velocity weight to absorb the extra velocity-growth factor.  That is the reason why we introduce into the $X(t)$ norm two energy functionals $\CE_{N_1,\ell_1,\la_0}(t)$ and $\CE_{N_0,\ell_0,\la_0}(t)$ with the approximate choice of $N$ and $\ell$. 

The second difficulty due to regularity-loss is the control of inner product terms
\begin{equation}
\notag
\sum_{|\al|=N_1}\lag \pa^\al E\cdot \xi\mu^{1/2},w_{-\ell_1,\la_0}^2\pa^\al f\rag,
\end{equation}
which arises from the weighted estimate on $\pa^\al f$. To deal with it, we use the time-weighted energy estimate with the time rate of negative power; the similar technique has been used in \cite{HK}. In fact, starting from the Lyapunov inequality 
\begin{equation}
\notag
\frac{d}{dt}\CE_{N_1}(t)+\kappa \CD_{N_1}(t)\lesssim h.o.t.,
\end{equation}
where $h.o.t.$ denotes the high order terms only, it follows that
\begin{multline}
\notag
\frac{d}{dt} [(1+t)^{-\eps_0}\CE_{N_1}(t)]+\ka (1+t)^{-\eps_0}\CD_{N_1}(t)\\
+\frac{\eps_0}{(1+t)^{1+\eps_0}}\CE_{N_1}(t)\lesssim (1+t)^{-\eps_0}\times  \{h.o.t.\}.
\end{multline}
Therefore, provided that the term on the right is time integrable, one can recover the following new dissipation term
\begin{equation}
\notag
\frac{\eps_0}{(1+t)^{1+\eps_0}}\sum_{|\al|=N_1}\|\pa^\al (E,B)\|^2.
\end{equation}
Although this dissipation term is degenerate in large time, it is enough to control other trouble terms related to the $N_1$th-order derivatives of $(E,B)$.

We finally mention Theorem \ref{thm.lb} concerning the time-decay property of the linearized system. Theorem \ref{thm.lb} not only plays a key role of dealing with the dispersion of the macro component $Pf$ in $\R^3$ due to the degeneration of $L$, but also its proof, particularly the Lyapunov inequality \eqref{ad.p1}, fully reveals the optimal dissipative structure of the linearized system, which further motivates the design of the energy norm $X(t)$. We remark that inspired by \eqref{ad.p1}, it would be an interesting problem to consider the spectrum \cite{U74,DL} or Green's function \cite{LY-G} of such complex system.

The rest of the paper is arranged as follows. In Section \ref{sec2}, we list some known facts for the macro structure of the system and also the basic estimates on $L$ and $\Ga$. In Section \ref{sec3}, we obtain the time-decay property of the linearized homogeneous system. In Section \ref{sec4}, we present series of lemmas for the a priori estimates on the solution and finish the proof of Theorem \ref{thm.gl}  at the end.

\subsection{Notations}

Throughout this paper,  $C$  denotes
some generic positive (generally large) constant and $\ka$ denotes some generic
positive (generally small) constant, where both $C$ and
$\ka$ may take different values in different places.
$A\lesssim B$ means that  there is a generic constant $C>0$ such that $A\leqslant CB$. $A\sim B$ means $A\lesssim B$ and $B\lesssim A$.  We use $L^2$ to denote the usual Hilbert spaces $L^2=L^2_{x,\xi}$ or $L^2_x$ with the norm $\|\cdot\|$, and use
$\lag \cdot,\cdot\rag$ to denote the inner product over $L^2_{x,\xi}$ or $L^2_\xi$.
For $q\geq 1$,  the  mixed  velocity-space Lebesgue
space $Z_q=L^2_\xi(L^q_x)=L^2(\R^3_\xi;L^q(\R^3_x))$ is used.
For
multi-indices $\al=(\al_1,\al_2,\al_3)$ and
$\be=(\be_1,\be_2,\be_3)$,  $\pa^{\al}_\be=\pa_x^\al\pa_\xi^\be=\pa_{x_1}^{\al_1}\pa_{x_2}^{\al_2}\pa_{x_3}^{\al_3}
    \pa_{\xi_1}^{\be_1}\pa_{\xi_2}^{\be_2}\pa_{\xi_3}^{\be_3}.
$
The length of $\al$ is $|\al|=\al_1+\al_2+\al_3$ and similar for $|\be|$.

\section{Preliminary}\label{sec2}

\subsection{Macro structure}

Consider the following linearized Vlasov-Maxwell-Landau system with a non-homogeneous source $\sourceG=[\sourceG_+(t,x,\xi),\sourceG_-(t,x,\xi)]$:
\begin{equation}\label{lsns}
    \left\{\begin{array}{l}
  \dis     \pa_t f_\pm+\xi\cdot\na_x f_\pm \mp E\cdot \xi \mu^{1/2}+L_{\pm} f =\sourceG_\pm,\\
 \dis \pa_t E-\na_x\times B=-\int_{\R^3} \xi \mu^{1/2}(f_+-f_-)\,d\xi,\\
\dis \pa_t B+\na_x\times E =0,\\
\dis\na_x\cdot E =\int_{\R^3} \mu^{1/2} (f_+-f_-)\,d\xi,\ \ \na_x \cdot
B=0.
    \end{array}\right.
\end{equation}
Taking velocity integrations of the first equation of \eqref{lsns} with respect to the velocity moments
\begin{equation}
\notag
    \mu^{1/2},\ \  \xi_i  \mu^{1/2}, i=1,2,3,\ \  \frac{1}{6}(|\xi|^2-3)\mu^{1/2},
\end{equation}
one has
\begin{eqnarray*}
&& \pa_t a_\pm +\na_x\cdot b +\na_x \cdot \langle \xi
\mu^{1/2},\{I_\pm-P_\pm\} f\rangle=\langle\mu^{1/2},\sourceG_\pm\rangle,\\
&& \pa_t [b_i+  \langle \xi_i \mu^{1/2},\{I_\pm-P_\pm\} f\rangle
]+\pa^i (a_\pm+2c)\mp E_i
\\
&&\quad +\na_x\cdot \langle \xi\xi_i
\mu^{1/2},\{I_\pm-P_\pm\} f\rangle=\langle\xi_i  \mu^{1/2},\sourceG_\pm {-L_{\pm} f}\rangle,\\
&&\pa_t \left[c+ \frac{1}{6}\langle (|\xi|^2-3)\mu^{1/2},\{I_\pm-P_\pm\}
f\rangle
\right]+  \frac{1}{3} \na_x\cdot b
\\
&&\quad +
\frac{1}{6}\na_x\cdot \langle
 (|\xi|^2-3)\xi\mu^{1/2},\{I_\pm-P_\pm\} f\rangle
 =\frac{1}{6}\langle (|\xi|^2-3)\mu^{1/2},\sourceG_\pm  {-L_{\pm} f}\rangle,
\end{eqnarray*}
where $\pa^i=\pa_{x_i}$ denotes the spatial derivative with respect to $x_i$, and we have set $I=[I_+,I_-]$ with $I_\pm f=f_\pm$. As in \cite{DuanS1, DS-VMB}, we define the high-order moment
 functions $\highG(f_\pm)=(\highG_{ij}(f_\pm))_{3\times 3}$ and
$\highB(f_\pm)=(\highB_1(f_\pm),\highB_2(f_\pm),\highB_3(f_\pm))$ by
\begin{equation*}
  \highG_{ij}(f_\pm) = \langle(\xi_i\xi_j-1)\mu^{1/2}, f_\pm\rangle,\ \
  \highB_i(f_\pm)=\frac{1}{10}\langle(|\xi|^2-5)\xi_i\mu^{1/2},
  f_\pm\rangle.
\end{equation*}
Further taking velocity integrations of the first equation of \eqref{lsns} with respect to the above high-order
moments one has
\begin{eqnarray*}
&&\pa_t [\highG_{ii}( \{I_\pm-P_\pm\}f)+2c]+2\pa^i b_i
=\highG_{ii}(r_{\pm}+\sourceG_\pm),\notag\\
&&\pa_t \highG_{ij}( \{I_\pm-P_\pm\}f) +\pa^j b_i+\pa^i b_j +\na_x\cdot
\langle \xi\mu^{1/2},\{I_\pm-P_\pm\} f\rangle
\notag
\\
&&\hspace{4cm}=\highG_{ij}(r_{\pm}+\sourceG_\pm)+\langle\mu^{1/2},\sourceG_\pm\rangle,\ \ i\neq j,\\
&& \pa_t \highB_i( \{I_\pm-P_\pm\}f)+\pa^i c=\highB_i(r_{\pm}+\sourceG_\pm),
\end{eqnarray*}
where
$r_{\pm} =-\xi \cdot \na_x \{I_\pm-P_\pm\}f-L_{\pm} f$.

In particular, for the nonlinear system  \eqref{VM},  the non-homogeneous source $\sourceG=[\sourceG_+(t,x,\xi),\sourceG_-(t,x,\xi)]$ takes the form of
\begin{equation}
\notag
    \sourceG_\pm=\pm\frac{1}{2}E\cdot \xi f_\pm \mp (E+\xi\times B)\cdot \na_\xi f_\pm +\Ga_\pm(f,f).
\end{equation}
Then, it is straightforward to compute from integration by parts that
\begin{eqnarray*}
 \langle \mu^{1/2},\sourceG_\pm\rangle &=& 0,\\
   \langle \xi\mu^{1/2},\sourceG_\pm\rangle &=&\pm Ea_\pm \pm b\times B  {\pm} \langle\xi\mu^{1/2}, \{I_\pm-P_\pm\} f\rangle \times B\\
   && {+\langle \xi\mu^{1/2},\Ga_\pm (f,f)\rangle},\\
    \frac{1}{6}\langle (|\xi|^2-3)\mu^{1/2},\sourceG_\pm\rangle &=&\pm \frac{1}{3} b\cdot E \pm \frac{1}{3}
     \langle\xi\mu^{1/2}, \{I_\pm-P_\pm\} f\rangle\cdot E\\
     && {+\langle \frac{1}{6}(|\xi|^2-3)\mu^{1/2},\Ga_\pm (f,f)\rangle}.
\end{eqnarray*}

\subsection{Basic estimates on $L$ and $\Ga$}

In this section, we state two lemmas about some basic properties
of the
Landau operator.
Given a vector-valued function $u=(u_1,u_2,u_3)$, define
\begin{equation*}
   P_\xi u=\frac{\xi \otimes \xi }{|\xi|^2}u=\left\{\frac{\xi}{|\xi|}\cdot u \right\}\frac{\xi}{|\xi|} ,\quad i.e.,\ ({P}_\xi u)_i= \left\{\sum_{j=1}^3 \frac{\xi_j}{|\xi|}u_j\right\}\frac{\xi_i}{|\xi|},\ 1\leq i\leq 3.
\end{equation*}
Concerning the equivalent characterization of the dissipation rate and the
dissipative property of the linearized Landau operator, one has the following lemma; see \cite{DL, Mo,Guo-L} for the detailed proof.

\begin{lemma}[\cite{Guo-L}]\label{lem.Ld}
It holds that
\begin{multline}
\notag
    |f|_{\FD,\tau,\la}^2 \sim \left|(1+|\xi|)^{\frac{\ga}{2}} {P}_\xi \na_\xi f\right|_{\tau,\la}^2+\left|(1+|\xi|)^{\frac{\ga+2}{2}}\{{I}-{P}_\xi\}\na_\xi f\right|_{\tau,\la}^2\\
    +\left|(1+|\xi|)^{\frac{\ga+2}{2}} f\right|_{\tau,\la}^2.
\end{multline}
Moreover, $L$ is a nonnegative definite self-adjoint operator, and 
there exists $\kappa>0$ such that
\begin{equation*}
    \lag L f, f\rag \geq \kappa \left|\{I-P\} f\right|_\FD^2.
\end{equation*}

\end{lemma}

The following lemma states the weighted estimate on $L f$ and $\Ga(f,f)$.

\begin{lemma}[\cite{SG}]\label{lem.LGa}
There are  $\kappa>0$ and  $C>0$  such that
\begin{equation}
\notag
    \left\lag L f, w_{\tau,\la}^2(t,\xi) f\right\rag \geq \kappa |f|_{\FD,\tau,\la}^2-C\left|\chi_{\{|\xi|\leq 2C\}} f\right|_\tau^2.
\end{equation}
Moreover, let $|\be|>0$, $\tau=|\be|-\ell$ with $\ell\geq 0$. Then, for $\eta>0$ small enough, there is $C_\eta>0$ such that
\begin{multline}
\notag
    \left\lag \pa_\be L f, w_{\tau,\la}^2(t,\xi) \pa_\be f\right\rag \geq \kappa |\pa_\be f|_{\FD,\tau,\la}^2-\eta \sum_{|\be'|=|\be|} |\pa_{\be'} f|_{\FD,\tau,\la}^2\\
    -C_\eta \sum_{|\be'|<|\be|} |\pa_{\be'} f|_{\FD,|\be'|-\ell,\la}.
\end{multline}
And also, for $N\geq 8$, $|\al|+|\be|\leq N$ and $\tau=|\be|-\ell$ with $\ell\geq 0$, it holds that
\begin{multline}
\label{est.ga}
\dis \left\lag  \pa_\be^\al \Ga(f_1,f_2),w_{\tau,\la}^2(t,\xi)\pa_\be^\al f_3\right\rag\\
\dis \lesssim \sum_{\substack{|\al'|+|\be'|\leq N  \\
\be''\leq \be'\leq \be }}\left\{\left|\pa_{\be''}^{\al'}f_1\right|_{\tau}\left|\pa_{\be-\be'}^{\al-\al'}f_2\right|_{\FD,\tau,\la} \right.\\
\dis\left.+\left|\pa_{\be''}^{\al'} f_1\right|_{\FD,\tau}\left|\pa_{\be-\be'}^{\al-\al'}f_2\right|_{\tau,\la}\right\}\left|\pa_\be^\al f_3\right|_{\FD,\tau,\la}.
\end{multline}

\end{lemma}

 Notice that since the coefficient $\la/(1+t)^{\vth}$ in front of $\lag \xi\rag^2$ in the exponential part of the weight function $w_{\tau,\la}(t,\xi)$ is bounded uniformly in  $t\geq 0$, the proof of the above  lemma follows directly from the same argument used in \cite{SG}.

\section{Linearized analysis}\label{sec3}

Consider the Cauchy problem on the linearized Vlasov-Maxwell-Landau system with a
 source $\sourceG=\sourceG(t,x,\xi) = [\sourceG_+(t,x,\xi),
\sourceG_-(t,x,\xi)]$:
\begin{equation}\label{ls}
    \left\{\begin{array}{l}
  \dis     \pa_t f+\xi\cdot\na_x f - E\cdot \xi \mu^{1/2}q_1+L f =\sourceG,\\
 \dis \pa_t E-\na_x\times B=-\langle \xi \mu^{1/2}, f_+-f_-\rangle,\\
\dis \pa_t B+\na_x\times E =0,\\
\dis\na_x\cdot E =\langle\mu^{1/2}, f_+-f_-\rangle,\ \ \na_x \cdot
B=0,\\
\dis (f,E,B)|_{t=0}=(f_0,E_0,B_0),
    \end{array}\right.
\end{equation}
where initial data $[f_0,E_0,B_0]$ satisfies the
compatibility condition
\begin{equation}\label{comp.con}
\na_x\cdot E_0=\int_{\R^3}\mu^{1/2}(f_{0,+}-f_{0,-})\,d\xi,\ \ \na_x\cdot
B_0=0,
\end{equation}
and the source term $\sourceG$ is assumed to satisfy
\begin{equation}
\notag
\int_{\R^3}\mu^{1/2}(\sourceG_{+}-\sourceG_{-})\,d\xi=0.
\end{equation}
To consider the solution to the Cauchy problem \eqref{ls}, for simplicity, we denote $U=[f,E,B]$, $U_0=[f_0,E_0,B_0]$ so that one can formally write
\begin{equation}
\notag
U(t)=\semiG(t)U_0+\int_0^t\semiG(t-s)[\sourceG(s),0,0]\,ds,
\end{equation}
where $\semiG(t)$ is the linear solution operator for the Cauchy problem
on the linearized homogeneous system corresponding to  \eqref{ls} in the case when $\sourceG=0$.


For the linearized homogeneous system, we have the following result. 

\begin{theorem}\label{thm.lb}
Let $S=0$, and let  $[f,E,B]$ be the solution to
the Cauchy problem \eqref{ls}, \eqref{comp.con} of the linearized
homogeneous system. Define the velocity weight function $w=w(\xi)$ by
\begin{equation}\label{def.wli}
w(\xi)=\lag \xi\rag^{-\frac{\ga+2}{2}} 
\end{equation}
Then, for $\ell\geq 0$ and $\al\geq 0$ with $m=|\al|$,
\begin{eqnarray}
&& \|w^\ell \pa^\al f\|+\|\pa^\al (E,B)\|\notag\\
&&\lesssim (1+t)^{- \si_m} (\|w^{\ell+\ell_\ast^{low}} f_0\|_{Z_1}+\|(E_0,B_0)\|_{L^1_x})\notag\\
&&\qquad + (1+t)^{-j} (\|w^{\ell+\ell_{\ast}^{high}}\na_x^{{j+1}} \pa^\al f_0\|+\|\na_x^{{j+1}} \pa^\al (E_0,B_0)\|),\label{thm.lb.1}
\end{eqnarray}
where
\begin{equation}
\notag
\si_m=\frac{3}{4} +\frac{m}{2},\ \ell_\ast^{low}>2\si_m,\
\ell_{\ast}^{high}>0,\ 0\leq j<\ell_{\ast}^{high}.
\end{equation}
\end{theorem}

\begin{proof} It is divided by the following three steps. For brevity of presentation we would only sketch the proof  by clarifying how the known techniques in \cite{DYZ-s, DS-VMB, St-Op} can be adopted in the situation considered here.  

\medskip
\noindent{\bf Step 1.} We claim that there is a time-frequency interactive functional $\CE^{\rm int}(t,k)$ such that
\begin{equation}
\label{thm.lb.p0}
|\CE^{\rm int}(t,k)|\lesssim |\hat{f}|_{L^2}^2 +|[\hat{E},\hat{B}]|^2,
\end{equation}
and
\begin{multline}
\label{thm.lb.p1} 
\dis \pa_t\{|\hat{f}|_{L^2}^2
+|[\hat{E},\hat{B}]|^2 +\ka_0\, \Re\,\CE^{\rm int}(t,k)\}
+{\ka |\{I-P\}\hat{f}|_{{\bf D}}^2} \\
\dis+ \frac{\ka |k|^2}{1+|k|^2} (|\widehat{a_++a_-}|^2+|\hat{b}|^2+|\hat{c}|^2)+\ka |\widehat{a_+-a_-}|^2\\
\dis +\frac{\ka}{1+|k|^2}|\hat{E}|^2 + \frac{\ka |k|^2}{(1+|k|^2)^2}|\hat{B}|^2\leq 0,
\end{multline}
where $\ka_0>0$ is a small constant such that
\begin{equation}
\label{thm.lb.p1-0}  
|\hat{f}|_{L^2}^2 +|[\hat{E},\hat{B}]|^2 +\ka_0\,\Re\, 
\CE^{\rm int}(t,k)\sim |\hat{f}|_{L^2}^2 +|[\hat{E},\hat{B}]|^2.
\end{equation}
Compared to the corresponding estimate obtained in \cite{DS-VMB}, the main improvement in \eqref{thm.lb.p1} occurs to the coefficient in front of the term $|\hat{E}|^2$ in the dissipation rate.

\medskip

\noindent{\bf Step 2.} In this step, we follow the approach in \cite{St-Op} to carry out the velocity weighted energy estimates for the pointwise frequency variable. Starting from the micro equation
\begin{equation*}
\begin{split}
\partial_t\{I-P\}\hat{f}&+i\xi\cdot k\{I-P\}\hat{f}+ L\{I-P\}\hat{f}\\ =&-\{I-P\}(\widehat{E}\cdot \xi)\sqrt{\mu}q_1-\{I-P\}\left[i\xi\cdot k{P}\hat{f}\right]+{P}\left[i\xi\cdot k\{I-P\}\hat{f}\right],
\end{split}
\end{equation*}
one can verity  that for $\ell\geq 0$,
\begin{multline}\label{ft microzero energy2}
\partial_t\left|w^{\ell}\{I-P\}\hat{f}\right|_{L^2}^2\chi_{|k|\leq 1}
+\ka{\left|w^{\ell}\{I-P\}\hat{f}\right|_{\FD}^2}\chi_{|k|\leq 1}\\
\lesssim \frac{1}{1+|k|^2}|\widehat{E}|^2+\frac{|k|^2}{1+|k|^2}|(\widehat{a_++a_-},\hat{b},\hat{c})|^2+|\widehat{a_+-a_-}|^2\\
+\left|w^{-1}\{I-P\}\hat{f}\right|_{L^2}^2.
\end{multline}
In a similar way, starting with the first equation of (\ref{ls}),
the direct velocity weighted energy estimates for the pointwise frequency variable  also gives
 that for $\ell\geq 0$,
\begin{multline}\label{ft spatialzero energy2}
\frac{1}{1+|k|^2}\partial_t\left|w^\ell\hat{f}\right|_{L^{2}}^2\chi_{|k|\geq 1}
+\frac{\ka}{1+|k|^2}{\left|w^{\ell}\{I-P\}\hat{f}\right|_{\FD}^2}\chi_{|k|\geq 1}\\
\lesssim \frac{1}{1+|k|^2}|\widehat{E}|^2+\frac{|k|^2}{1+|k|^2}|(\widehat{a_++a_-},\hat{b},\hat{c})|^2+|\widehat{a_+-a_-}|^2\\
+\left|w^{-1}\{I-P\}\hat{f}\right|_{L^2}^2.
\end{multline}
Now, by choosing $\kappa_1,\ka_2>0$ properly small, the linear combination $\eqref{thm.lb.p1}+\ka_2\times \eqref{ft microzero energy2}+\kappa_1\times\eqref{ft spatialzero energy2}$ yields that for $\ell\geq 0$,
\begin{equation}\label{ad.p1}
\pa_t M_\ell(t,k) +\ka D_\ell(t,k)\leq 0,
\end{equation}
where $M_\ell(t,k)$ and $D_\ell(t,k)$ are given by
\begin{eqnarray*}
   M_\ell(t,k)&=&\|\hat{f}\|_{L^2}^2+|[\hat{E},\hat{B}]|^2+\ka_0\,\Re\, \CE^{\rm int}(t,k)\\
   && +\kappa_2\left|w^\ell\{I-P\}\hat{f}\right|_{L^2}\,\chi_{|k|\leq 1}
    + \frac{\kappa_1}{1+|k|^2}\left|w^\ell\hat{f}\right|_{L^2}^2\chi_{|k|\geq 1},\nonumber\\
   D_\ell(t,k)&=&|\{I-P\}\hat{f}|_{{\bf D}}^2
   + \frac{1}{1+|k|^2}{\left|w^{\ell}\{I-P\}\hat{f}\right|_{\FD}^2}\\
&&\dis+ \frac{|k|^2}{1+|k|^2} (|\widehat{a_++a_-}|^2+|\hat{b}|^2+|c|^2)+ |\widehat{a_+-a_-}|^2\\
&&\dis +\frac{1}{1+|k|^2}|\hat{E}|^2 + \frac{|k|^2}{(1+|k|^2)^2}|\hat{B}|^2.
\end{eqnarray*}
Here, notice that for any $\ell$,
\begin{equation}
\notag
{\left|w^{\ell}\{I-P\}\hat{f}\right|_{\FD}\gtrsim  \left|w^{\ell-1}\{I-P\}\hat{f}\right|_{L^2}.}
\end{equation}
Moreover, set the frequency function $
\rho(k)=|k|^2/(1+|k|^2)^2$.
Then, by considering the $[1+\eps\rho (k)t]^J$-weighted estimate on \eqref{ad.p1} and  also applying the iterative technique developed in \cite{DYZ-s}, one has that whenever $\ell\geq 0$,
\begin{equation}\label{ELI}
   M_\ell(t,k)\lesssim   [1+\eps \rho(k)t]^{-J} M_{\ell+J+p-1}(0,k),
\end{equation}
holds true for any $t\geq 0$, $k\in \R^3$, where the parameters $p$, $\eps$ and $J$ with
\begin{equation*}
p>1,\  0<\eps\leq 1, \ J>0,\  C_1\eps J \leq \frac{\ka}{4}.
\end{equation*}
are still to be chosen.

\medskip

\noindent{\bf Step 3.} For $\ell\geq 0$, define
\begin{equation}
\label{thm.lb.p8}
\widetilde{M}_\ell(t,k)=|w^\ell \hat{f}|_{L^2}^2+|[\hat{E},\hat{B}]|^2.
\end{equation}
Take $\al\geq 0$ with $m=|\al|$.
We use the splitting
\begin{equation}
\notag
\int_{{\R}^3} |k^\al|^2 \widetilde{M}_\ell(t,k)\,dk=\left(\int_{|k|\leq 1 }+\int_{|k|\geq 1} \right)|k^\al|^2 \widetilde{M}_\ell(t,k)\,dk.
\end{equation}
By the estimate (\ref{ELI}), one has
\begin{multline*}
\dis \int_{|k|\leq 1}|k^\al|^2 \widetilde{M}_\ell(t,k)\,dk\lesssim \int_{|k|\leq 1}|k^\al|^2 {M}_\ell(t,k)\,dk \\
\leq \int_{|k|\leq 1} |k|^{2m} [1+\eps \rho(k)t]^{-J} M_{\ell+J+p-1}(0,k)\,dk\\
\leq \int_{|k|\leq 1} |k|^{2m}\left[1+\frac{\eps}{4}|k|^2 t\right]^{-J} dk\ \sup_{k\in\R^3} M_{\ell+J+p-1} (0,k).
\end{multline*}
By letting
$
2J-2m>3$, i.e. $J>m+\frac{3}{2}=2\si_m$,
we arrive at 
\begin{equation}
\notag
 \int_{|k|\leq 1}|k^\al|^2 \widetilde{M}_\ell (t,k)\, dk
 \lesssim (1+t)^{- \left(\frac{3}{2}+m\right)}\left(\left\|w^{\ell+J+p-1} f_0\right\|_{Z_1}^2 +\|(E_0,B_0)\|_{L^1_x}^2\right).
\end{equation}
Here, notice that  for any given $\ell^{low}_\ast>2\si_m$, one can take proper constants $J>2\si_m$ and $p>1$ such that 
$
\ell^{low}_\ast=J+p-1.
$
Hence, for the low frequency part,
\begin{equation}
\label{thm.lb.p9}
\int_{|k|\leq 1}|k^\al|^2 \widetilde{M}_\ell (t,k)\, dk\\
 \lesssim (1+t)^{- 2\si_m}\left(\left\|w^{\ell+\ell^{low}_\ast} f_0\right\|_{Z_1}^2 +\|(E_0,B_0)\|_{L^1_x}^2\right),
\end{equation}
for $\ell^{low}_\ast>2\si_m$.
To estimate the high-frequency part, let $j\geq 0$ to be chosen, and one has
\begin{multline}
\notag
\int_{|k|\geq 1}|k^\al|^2  \widetilde{M}_\ell (t,k)\, dk\\
\lesssim \int_{|k|\geq 1} |k^\al|^2 (1+|k|^2)M_{\ell}(t,k)\,dk
\lesssim
\int_{|k|\geq 1} |k^\al|^2 |k|^2M_{\ell}(t,k)\,dk\\
\lesssim
\int_{|k|\geq 1} |k^\al|^2 |k|^2[1+\eps \rho(k)t]^{-J} M_{\ell+J+p-1}(0,k)\,dk,
\end{multline}
which further implies
\begin{eqnarray*}
&&\int_{|k|\geq 1}|k^\al|^2  \widetilde{M}_\ell (t,k)\, dk\\
&&\lesssim \int_{|k|\geq 1}  \left[1+\frac{\eps t}{4|k|^2}\right]^{-J}  |k^\al|^2 |k|^2M_{\ell+J+p-1} (0,k)\,dk\\
&&\lesssim \int_{\R^3}|k^\al|^2 |k|^{2(j+1)} M_{\ell+J+p-1} (0,k)\,dk
\cdot  \sup_{|k|\geq 1} \left\{ \left[1+\frac{\eps t}{4|k|^2}\right]^{-J} \frac{1}{|k|^{2j}}\right\}.
\end{eqnarray*}
Notice that 
\begin{equation*}
\sup_{|k|\geq 1} \left\{ \left[1+\frac{\eps t}{4|k|^2}\right]^{-J} \frac{1}{|k|^{2j}}\right\}
\lesssim (1+t)^{-j},
\end{equation*}
as long as
$j\leq J$.
Thus, we have
\begin{multline}
\notag
 \int_{|k|\geq 1}|k^\al|^2 \widetilde{M}_\ell (t,k)\,dk\\
 \lesssim  (1+t)^{-j} \left(\left\|w^{\ell+J+p-1}\na_x^{j+1}\pa^\al f_0\right\|^2+\left\|\na_x^{j+1} \pa^\al (E_0,B_0)\right\|^2\right).
\end{multline}
Here, notice that for any given $\ell_\ast^{high}>0$ and $0\leq j<\ell_\ast^{high}$, one can choose again proper constants $J> j$ and $p>1$ such that 
$
\ell^{high}_\ast=J+p-1.
$
Hence, it follows that 
\begin{multline}
\label{thm.lb.p10}
 \int_{|k|\geq 1}|k^\al|^2 \widetilde{M}_\ell (t,k)\,dk\\
 \lesssim  (1+t)^{-j} \left(\left\|w^{\ell+\ell_\ast^{high}}\na_x^{j+1}\pa^\al f_0\right\|^2+\left\|\na_x^{j+1} \pa^\al (E_0,B_0)\right\|^2\right),
\end{multline}
for $\ell_\ast^{high}>0$ and $0\leq j<\ell_\ast^{high}$. Therefore, in terms of \eqref{thm.lb.p8}, the desired time-decay estimate \eqref{thm.lb.1} follows from the combination of \eqref{thm.lb.p9} and \eqref{thm.lb.p10}. This completes the proof of Theorem \ref{thm.lb}.
\end{proof}

\section{Global a priori estimates}\label{sec4}

In this section we will prove Theorem \ref{thm.gl}. The key point is to  deduce the uniform-in-time {\it a priori} estimates of solutions to the Vlasov-Maxwell-Landau system
\begin{equation}\label{ns}
    \left\{\begin{array}{l}
  \dis     \pa_t f+\xi\cdot\na_x f - E\cdot \xi \mu^{1/2}q_1+L f =\sourceG,\\
 \dis \pa_t E-\na_x\times B=-\langle \xi \mu^{1/2}, f_+-f_-\rangle,\\
\dis \pa_t B+\na_x\times E =0,\\
\dis\na_x\cdot E =\langle\mu^{1/2}, f_+-f_-\rangle,\ \ \na_x \cdot
B=0,
    \end{array}\right.
\end{equation}
where the nonlinear term is given by
\begin{equation}
\label{def.S}
S=\Ga(f,f)+\frac{1}{2}q_0 E\cdot \xi f -q_0 (E+\xi\times B)\cdot \na_\xi f.
\end{equation}
For that, let $(f,E,B)$ be a smooth solution to \eqref{ns} over the time interval $0\leq t\leq T$ with initial data $(f_0,E_0,B_0)$ for $0<T\leq \infty$, and further suppose that  $(f,E,B)$ satisfies
\begin{equation}
\label{apX}
X(t)\leq \de^2,
\end{equation}
where $X(t)$ is given in \eqref{def.X} and the constant $\de>0$ is sufficiently small. We point out that throughout this section, in the case when an undetermined energy functional $\CE_{N,\ell,\la}(t)$ appears on the right-hand side of inequalities, it is always understood to take exactly the right-hand expression of \eqref{def.e}.

\subsection{Macro dissipation}
Basing on the previous work \cite{DuanS1} and \cite{DS-VMB}, it is a standard process to obtain in terms of the following lemma the macro dissipation $\CD_{N,mac}(t)$ defined by
\begin{multline}
\notag
\CD_{N,mac}(t)=\sum_{|\al|\leq N-1} \|\na_x \pa^\al (a_\pm, b,c)\|^2
+ \|a_+-a_-\|^2\\
+ \|E\|_{H^{N-1}}^2 +\|\na_x B\|_{H^{N-2}}^2.
\end{multline}
The details for the proof of this lemma are omitted for brevity. 

\begin{lemma}\label{lem.mad}
For any integer $N$ with $8\leq N\leq N_1$, there is an interactive energy functional $\CE_N^{int}(t)$ such that 
\begin{equation}
\notag
|\CE_N^{int}(t)|\lesssim \sum_{|\al|\leq N}(\|\pa^\al f\|^2+\|\pa^\al (E,B)\|^2)
\end{equation}
and
\begin{equation}
\notag
\frac{d}{dt} \CE_N^{int}(t)+\ka \CD_{N,mac}(t)\lesssim  \sum_{
|\al|\leq N }\|\pa^\al \{I-P\} f\|_{\FD}^2+\CE_{N}(t)\CD_{N}(t),
\end{equation}
for $0\leq t\leq T$.
\end{lemma}

\subsection{Lyapunov inequality for $\CE_{N_1}(t)$}

In this section we derive the basic energy estimates on $\CE_{N,\ell,\la}(t)$ in the case of $N=N_1$ and $\ell=\la=0$. Notice that $N_1$ is the highest order, and regarding $\pa_\be^\al f$, the weight function takes the form of $w_{-|\be|}(\xi)$ independent of time.

\begin{lemma}\label{lem.n1}
There is an energy functional $\CE_{N_1}(t)$ such that 
\begin{equation}
\label{lem.n1.1}
\frac{d}{dt}\CE_{N_1}(t)+\kappa\CD_{N_1}(t)\lesssim \frac{\de}{(1+t)^{1+\vth}}\CD_{N_1,\ell_1,\la_0}(t)+\CE_{N_1}(t)\CE_{N_0,\ell_0-1,\la_0}(t),
\end{equation}
for $0\leq t\leq T$.
\end{lemma}

\begin{proof}
First of all, it is straightforward to establish the energy identities
\begin{equation}
\label{lem.n1.p1}
\frac{1}{2}\frac{d}{dt} (\|f\|^2+\|(E,B)\|^2) + \lag L f, f\rag =\lag S, f\rag,
\end{equation}
and
\begin{multline}
\label{lem.n1.p2}
\frac{1}{2}\frac{d}{dt}\sum_{1\leq |\al|\leq N_1} (\|\pa^\al f\|^2+\|\pa^\al (E,B)\|^2) + \sum_{1\leq |\al|\leq N_1} \lag L\pa^\al f,\pa^\al f\rag \\
=\sum_{1\leq |\al|\leq N_1} \lag \pa^\al S,\pa^\al f\rag.
\end{multline}
By applying the micro projection $I-P$ to the first equation of \eqref{ns}, it can be rewritten as 
\begin{eqnarray}
\pa_t \{I-P\}f +\xi\cdot \na_x \{I-P\}f &-&E\cdot \xi \mu^{1/2} +L f \notag\\
&&= \{I-P\}S+P\xi\cdot \na_x f -\xi \cdot \na_x P f.\label{ns.mic}
\end{eqnarray}
After acting $\pa^\al_\be$ with $|\al|+|\be|\leq N_1$ and $|\be|\geq 1$ to the above equation and further multiplying it by $\lag \xi\rag^{2(\ga+2)|\be|}\pa_\be^\al \{I-P\}f$, the direct energy estimate gives the identity
\begin{eqnarray*}
&&\frac{1}{2}\|\pa_\be^\al \{I-P\} f\|_{|\be|}^2+\lag \pa_\be^\al Lf,\lag \xi\rag^{2(\ga+2)|\be|}\pa_\be^\al \{I-P\}f\rag\\
&& =\lag \pa^\al_\be \{-\xi\cdot \na_x \{I-P\}f+E\cdot \xi \mu^{1/2}\},\lag \xi\rag^{2(\ga+2)|\be|}\pa_\be^\al \{I-P\}f\rag\\
&&\quad+\lag \pa^\al_\be \{\{I-P\}S+P\xi\cdot \na_x f -\xi \cdot \na_x P f\},\lag \xi\rag^{2(\ga+2)|\be|}\pa_\be^\al \{I-P\}f\rag, 
\end{eqnarray*}
which from Lemma \ref{lem.LGa} implies
\begin{eqnarray}
&\dis  \frac{1}{2}\frac{d}{dt}\sum_{m=1}^{N_1}C_m\sum_{\substack{|\be|= m \notag\\
|\al|+|\be|\leq N_1 }}\|\pa_\be^\al \{I-P\} f\|_{|\be|}^2 +\ka \sum_{\substack{|\be|\geq 1 \\
|\al|+|\be|\leq N_1 }}\|\pa_\be^\al \{I-P\} f\|_{\FD, |\be|}^2 \notag\\
&\dis \lesssim \sum_{\substack{|\be|= 0 \\
|\al|+|\be|\leq N_1 }}\|\pa_\be^\al \{I-P\} f\|_{\FD, |\be|}^2+ \sum_{|\al|\leq N_1-1} (\|\na_x\pa^\al (a_\pm,b,c)\|^2+\|\pa^\al E\|^2)\notag\\
&\dis +\sum_{m=1}^{N_1}C_m\sum_{\substack{|\be|= m \\
|\al|+|\be|\leq N_1 }}\lag \pa^\al_\be \{I-P\}S,\lag \xi\rag^{2(\ga+2)|\be|}\pa_\be^\al \{I-P\}f\rag.\label{lem.n1.p3}
\end{eqnarray}
Moreover, from Lemma \ref{lem.mad} as well as \eqref{apX},
\begin{equation}
\label{lem.n1.p4}
\frac{d}{dt} \CE_{N_1}^{int}(t)+\ka \CD_{N_1,mac}(t)\lesssim  \sum_{
|\al|\leq N_1 }\|\pa^\al \{I-P\} f\|_{\FD}^2+\de^2\CD_{N_1}(t).
\end{equation}
Then, since $\de>0$ can be small enough, the proper linear combination of \eqref{lem.n1.p1}, \eqref{lem.n1.p2}, \eqref{lem.n1.p3} and \eqref{lem.n1.p4} implies that there is an energy functional $\CE_{N_1}(t)$ satisfying \eqref{def.e} such that 
\begin{equation}
\label{lem.n1.p5}
\frac{d}{dt}\CE_{N_1}(t) +\kappa \CD_{N_1}(t)\lesssim \CI^{(1)}_{N_1}(t),
\end{equation}
where
\begin{multline}
\notag
\CI^{(1)}_{N_1}(t)=\lag S,f\rag +\sum_{1\leq |\al|\leq N_1} \lag \pa^\al S,\pa^\al f\rag\\
+\sum_{m=1}^{N_1}C_m\sum_{\substack{|\be|= m \\
|\al|+|\be|\leq N_1 }}\lag \pa^\al_\be \{I-P\}S,\lag \xi\rag^{2(\ga+2)|\be|}\pa_\be^\al \{I-P\}f\rag.
\end{multline}
Finally, we claim that
\begin{multline}
\label{lem.n1.p6}
\CI^{(1)}_{N_1}(t)\lesssim  \CE_{N_1}^{1/2}(t)\CD_{N_1}(t) +\frac{\de}{(1+t)^{1+\vth}}\CD_{N_1,\ell_1,\la_0}(t)\\
+\CE_{N_1}^{1/2}(t) \CE_{N_0,\ell_0-1,\la_0}^{1/2}(t)\CD_{N_1}^{1/2}(t).
\end{multline}
Therefore, \eqref{lem.n1.1} follows from plugging \eqref{lem.n1.p6} into \eqref{lem.n1.p5} and applying \eqref{apX} and the Cauchy-Schwarz inequality. This then completes the proof of Lemma \ref{lem.n1}.
\end{proof}

\begin{proof}[Proof of \eqref{lem.n1.p6}]
We first consider the estimate of $\CI^{(1)}_{N_1}(t)$ corresponding to $\Ga(f,f)$ in the nonlinear term $S$. As in \cite{Guo-L}, by decomposing $f$ as $Pf+\{I-P\}f$ and using Lemma \ref{lem.LGa}, it directly follows that it is bounded up to a generic constant by $\CE_{N_1}^{1/2}(t)\CD_{N_1}(t)$.
Recall from the definition of $X(t)$,
\begin{equation}
\notag
\|\na_x(E,B)\|_{H^{N_0-1}} \leq \frac{X^{1/2}(t)}{(1+t)^{1+\vth}}\leq \frac{\de}{(1+t)^{1+\vth}}.
\end{equation}
For the zero-order term related to the electromagnetic field, it holds that
\begin{eqnarray*}
&&\lag \frac{1}{2}q_0E\cdot \xi f-q_0 (E+\xi\times B)\cdot \na_\xi f,f\rag=\lag \frac{1}{2}q_0E\cdot \xi f,f\rag\\
&&\lesssim\iint_{\R^3\times \R^3} |E|\cdot |\xi|(|Pf|^2+|\{I-P\}f|^2)\,dxd\xi\\
&&\lesssim \|E\|\cdot \|(a,b,c)\|_{L^\infty}\|(a,b,c)\|+\|E\|_{L^\infty} \iint_{\R^3\times \R^3}|\xi|\cdot |\{I-P\}f|^2\,dxd\xi\\
&&\lesssim\CE_{N_1}^{1/2}(t)\CD_{N_1}(t)+\frac{\de}{(1+t)^{1+\vth}} \CD_{N_1,\ell_1,\la_0}(t).
\end{eqnarray*}
For the $\pa^\al$ derivative term related to $(E,B)$ with $1\leq |\al|\leq N_1$, one has
\begin{eqnarray*}
&&\lag \pa^\al (E\cdot \xi f),\pa^\al f\rag=\sum_{\al_1\leq \al}C_{\al_1}^\al \lag \pa^{\al_1}E\cdot \xi \pa^{\al-\al_1} f,\pa^\al f\rag\\
&& \lesssim \sum_{|\al_1|\leq N_1/2,\al_1\neq \al}\|\pa^{\al_1} E\|_{L^\infty} \left\||\xi|\lag \xi \rag^{\frac{\ga+2}{2}} \pa^{\al-\al_1}f\right\|\cdot \|\lag \xi\rag^{\frac{\ga+2}{2}}\pa^\al f\|\\
&&\quad +\sum_{|\al_1|>N_1/2\, \text{or}\, \al_1=\al}\|\pa^{\al_1} E\|\cdot \sup_{x}\left||\xi|\lag \xi \rag^{\frac{\ga+2}{2}} \pa^{\al-\al_1}f\right|_{L^2_\xi} \|\lag \xi\rag^{\frac{\ga+2}{2}}\pa^\al f\|\\
&&\lesssim \frac{\de}{(1+t)^{1+\vth}}\CD_{N_1,\ell_1,\la_0}(t)
+\CE_{N_1}^{1/2}(t) \CE_{N_0,\ell_0-1,\la_0}^{1/2}(t)\CD_{N_1}^{1/2}(t),
\end{eqnarray*}
where the Sobolev inequality $\|g\|_{L^\infty}\leq C\|\na_x g\|_{H^1}$ for any function $g=g(x)\in H^2$ has been used, and we also have used the choice of $N_1=\frac{3}{2}N_0$ and $\ell_1=\frac{1}{2}\ell_0$ with $N_0$ and $\ell_0$ properly large. And in a similar way, it holds that
\begin{multline}
\notag
\lag \pa^\al \{(\xi\times B)\cdot \na_\xi f\},\pa^\al f\rag =\sum_{0<\al_1\leq \al}C_{\al_1}^\al \lag (\xi\times \pa^{\al_1}B)\cdot \na_\xi \pa^{\al-\al_1}f,\pa^\al f\rag\\
\lesssim \frac{\de}{(1+t)^{1+\vth}}\CD_{N_1,\ell_1,\la_0}(t)
+\CE_{N_1}^{1/2}(t) \CE_{N_0,\ell_0-1,\la_0}^{1/2}(t)\CD_{N_1}^{1/2}(t).
\end{multline}
The completely same estimate holds true for the $\pa_\be^\al$ derivative term related to $(E,B)$ with $|\al|+|\be|\leq N_1$ and $|\be_1|\geq 1$, through observing
\begin{multline}
\notag
\{I-P\}\{\frac{1}{2}q_0E\cdot \xi f-q_0 (E+\xi\times B)\cdot \na_\xi f\}\\
=\frac{1}{2}q_0E\cdot \xi \{I-P\}f-q_0 (E+\xi\times B)\cdot \na_\xi \{I-P\}f\\
+\frac{1}{2}q_0E\cdot \xi Pf-q_0 (E+\xi\times B)\cdot \na_\xi Pf\\
-P\{\frac{1}{2}q_0E\cdot \xi f-q_0 (E+\xi\times B)\cdot \na_\xi f\}.
\end{multline}
Therefore, \eqref{lem.n1.p6} follows by collecting these estimates.
\end{proof}

\subsection{Lyapunov inequality for $\CE_{N_1,\ell_1,\la_0}(t)$}

In this section we turn to the weighted energy estimates on $\CE_{N_1,\ell_1,\la_0}(t)$. Notice that due to the regularity-loss property of the whole system, the weighted high-order energy functional $\CE_{N_1,\ell_1,\la_0}(t)$ can not be  bounded uniformly in time; it has been actually seen from the proof of Lemma \ref{thm.lb} for the linearized analysis, cf.~\eqref{ad.p1}. Instead we may expect that $\CE_{N_1,\ell_1,\la_0}(t)$ increases in time. Another trouble arises from the weighted estimate on derivatives of the highest order $N_1$ for the linear term $E\cdot \xi \mu^{1/2}$. However, these difficulties will be resolved by considering the time weighted estimate with the time rate of negative power.  In addition, although the velocity growth in the nonlinear term containing the electromagnetic field could be dealt with through the time-dependent exponential factor in the weight function, it is impossible when the  electromagnetic field gains the differentiation of higher orders since they again could not decay in time. Thus, to overcome it, one has to refine the nonlinear estimates in order to use the time-decay property of the lower-order energy functional $\CE_{N_0,\ell_0-1,\la_0}$ with the higher-order algebraic velocity weight.

\begin{lemma}\label{lem.n1w}
There is an energy functional $\CE_{N_1,\ell_1,\la_0}(t)$ such that 
\begin{multline}
\label{lem.n1w.1}
\frac{d}{dt}\CE_{N_1,\ell_1,\la_0}(t)+\kappa\CD_{N_1,\ell_1,\la_0}(t)\lesssim \CE_{N_1,\ell_1,\la_0}(t)\CE_{N_0,\ell_0-1,\la_0}(t)\\
+ \sum_{|\al|=N_1} \lag \pa^\al E\cdot \xi \mu^{1/2},w_{-\ell_1,\la_0}^2 \pa^\al f\rag,
\end{multline}
for $0\leq t\leq T$.
\end{lemma}

\begin{proof}
Starting from the first equation of \eqref{ns}, the energy estimate on $ \pa^\al f$ with $1\leq |\al|\leq N_1$ weighted by the time-velocity dependent function  $w_{-\ell_1,\la_0}=w_{-\ell_1,\la_0}(t,\xi)$ gives
\begin{multline}
\frac{1}{2}\frac{d}{dt} \sum_{1\leq |\al|\leq N_1} \|\pa^\al f\|_{-\ell_1,\la_0}^2+ \sum_{1\leq |\al|\leq N_1} \lag L\pa^\al f,w_{-\ell_1,\la_0}^2\pa^\al f \rag\\
+ \frac{\vth \la_0}{(1+t)^{1+\vth}} \|\lag \xi \rag \pa^\al f \|_{-\ell_1,\la_0}^2= \sum_{1\leq |\al|\leq N_1} \lag \pa^\al S,w_{-\ell_1,\la_0}^2\pa^\al f \rag\\
+ \sum_{1\leq |\al|\leq N_1} \lag\pa^\al E\cdot \xi \mu^{1/2},w_{-\ell_1,\la_0}^2\pa^\al f \rag.
\label{lem.n1w.p1}
\end{multline}
Similarly, from \eqref{ns.mic}, one has the weighted energy estimate on $\{I-P\}f$ 
\begin{eqnarray}
&&\dis  \frac{1}{2}\frac{d}{dt} \|\{I-P\}f\|_{-\ell_1,\la_0}^2+\lag Lf,w_{-\ell_1,\la_0}^2 \{I-P\}f \rag\notag\\
&&\quad\dis + \frac{\vth \la_0}{(1+t)^{1+\vth}} \|\lag \xi \rag \{I-P\}f \|_{-\ell_1,\la_0}^2\notag\\
&&=\lag \{I-P\}S,w_{-\ell_1,\la_0}^2\{I-P\} f \rag+\lag E\cdot \xi \mu^{1/2},w_{-\ell_1,\la_0}^2\{I-P\} f \rag\notag\\
&&\quad \dis+\lag P\xi\cdot \na_x f -\xi \cdot \na_x P f,w_{-\ell_1,\la_0}^2\{I-P\} f \rag,
\label{lem.n1w.p2}
\end{eqnarray}
and the weighted energy estimate on $\{I-P\}\pa_\be^\al f$ with $|\al|+|\be|\leq N_1$ and $|\be|\geq 1$ 
\begin{eqnarray}
&\dis  \frac{1}{2}\frac{d}{dt}\sum_{m=1}^{N_1}C_m\sum_{\substack{|\be|= m \\
|\al|+|\be|\leq N_1 }}\|\pa_\be^\al \{I-P\} f\|_{|\be|-\ell_1,\la_0}^2 \notag\\
&\dis +\ka \sum_{\substack{|\be|\geq 1 \\
|\al|+|\be|\leq N_1 }}\left(\|\pa_\be^\al \{I-P\} f\|_{\FD, |\be|-\ell_1,\la_0}^2 +\frac{\la_0}{(1+t)^{1+\vth}} \|\lag \xi \rag \pa_\be^\al \{I-P\}f\|\right)\notag\\
&\dis \lesssim \sum_{\substack{|\be|= 0 \\
|\al|+|\be|\leq N_1 }}\|\pa_\be^\al \{I-P\} f\|_{\FD, |\be|-\ell_1,\la_0}^2+ \sum_{|\al|\leq N_1-1} (\|\na_x\pa^\al (a_\pm,b,c)\|^2+\|\pa^\al E\|^2)\notag\\
&\dis +\sum_{m=1}^{N_1}C_m\sum_{\substack{|\be|= m \\
|\al|+|\be|\leq N_1 }}\lag \pa^\al_\be \{I-P\}S,w_{|\be|-\ell_1,\la_0}^2\pa_\be^\al\{I-P\} f\rag.
\label{lem.n1w.p3}
\end{eqnarray}
Then, the proper linear combination of  \eqref{lem.n1.p1}, \eqref{lem.n1.p2}, \eqref{lem.n1w.p1},  \eqref{lem.n1w.p2} and \eqref{lem.n1w.p3} implies that there is an energy functional $\CE_{N_1,\ell_1,\la_0}(t)$ satisfying \eqref{def.e} such that 
\begin{multline}
\label{lem.n1w.p4}
\frac{d}{dt}\CE_{N_1,\ell_1,\la_0}(t) +\ka \CD_{N_1,\ell_1,\la_0}(t)\lesssim \CI^{(2)}_{N_1,\ell_1,\la_0}(t) \\
+ \sum_{ |\al|= N_1} \lag\pa^\al E\cdot \xi \mu^{1/2},w_{-\ell_1,\la_0}^2\pa^\al f \rag.
\end{multline}
where
\begin{eqnarray}
 \CI^{(2)}_{N_1,\ell_1,\la_0}(t) &=&\lag S, f\rag + \sum_{1\leq |\al|\leq N_1}\lag \pa^\al S,\pa^\al f\rag \notag\\
&& +\lag \{I-P\}S, w_{-\ell_1,\la_0}^2\{I-P\}f\rag+\sum_{1\leq |\al|\leq N_1} \lag \pa^\al S,w_{-\ell_1,\la_0}^2\pa^\al f \rag\notag\\
&& +\sum_{m=1}^{N_1}C_m\sum_{\substack{|\be|= m \\
|\al|+|\be|\leq N_1 }}\lag \pa^\al_\be \{I-P\}S,w_{|\be|-\ell_1,\la_0}^2\pa_\be^\al\{I-P\} f\rag.\label{def.i2}
\end{eqnarray}
We now claim that 
\begin{eqnarray}
&\dis \CI^{(2)}_{N_1,\ell_1,\la_0}(t) \lesssim \CE_{N_0,\ell_0,\la_0}^{1/2}(t)\CD_{N_1,\ell_1,\la_0}(t)
 +\CE_{N_1,\ell_1,\la_0}^{1/2}(t)\CE_{N_0,\ell_0-1,\la_0}^{1/2}(t)\CD_{N_1,\ell_1,\la_0}^{1/2}(t)\notag\\
&\dis +\frac{1}{\la_0} (1+t)^{1+\vth}\|\na_x (E,B)\|_{H^{N_0-1}}  \CD_{N_1,\ell_1,\la_0}(t).
\label{lem.n1w.p5}
\end{eqnarray}
Notice that the first and third terms on the right are bounded up to a generic constant by $\de\CD_{N_1,\ell_1,\la_0}(t)$ due to the definition of $X(t)$ and the assumption $X(t)\leq \de^2$. This hence simplifies \eqref{lem.n1w.p5} as
\begin{equation}
\label{lem.n1w.p6}
 \CI^{(2)}_{N_1,\ell_1,\la_0}(t) \lesssim \CE_{N_1,\ell_1,\la_0}^{1/2}(t)\CE_{N_0,\ell_0-1,\la_0}^{1/2}(t)\CD_{N_1,\ell_1,\la_0}^{1/2}(t)
 +\de\CD_{N_1,\ell_1,\la_0}(t).
\end{equation}
Therefore, since $\de>0$ is small enough, by applying the Cauchy-Schwarz inequality to the  first term on the right-hand side of  \eqref{lem.n1w.p6} and then putting it into \eqref{lem.n1w.p4}, \eqref{lem.n1w.1} follows. This completes the proof of Lemma \ref{lem.n1w}.
\end{proof}

\begin{proof}[Proof of \eqref{lem.n1w.p5}]
For brevity, we only present the estimate of the fourth term on the right-hand side of \eqref{def.i2} since the estimate on other terms is simpler or follows in the completely same way. Take $\al$ with $1\leq |\al|\leq N_1$.
For the inner product term related to $\pa^\al \Ga(f,f)$, by using \eqref{est.ga} in Lemma \ref{lem.LGa}, considering the cases of $|\al'|\leq N_1/2$ and $|\al'|>N_1/2$, applying the Sobolev inequality $\|g\|_{L^\infty}\leq C \|\na_x g\|_{H^1}$ to the lower-order derivatives, and noticing the choice of $N_1=\frac{3}{2}N_0$ and $\ell_1=\frac{1}{2}\ell_0$ with $N_0$ and $\ell_0$ properly large, it follows that 
\begin{eqnarray*}
&&\lag \pa^\al\Ga(f,f), w_{-\ell_1,\la_0}^2\pa^\al f\rag\\
&&\lesssim
\{\CE_{N_0,\ell_0,\la_0}^{1/2}(t)\CD_{N_1,\ell_1,\la_0}^{1/2}(t)+\CE_{N_0,\ell_0-1,\la_0}^{1/2}(t)\CE_{N_1,\ell_1,\la_0}^{1/2}(t)\}\CD_{N_1,\ell_1,\la_0}^{1/2}(t)\\
&&\quad +\{\CE_{N_1,\ell_1,\la_0}^{1/2}(t)\CE_{N_0,\ell_0-1,\la_0}^{1/2}(t)+ \CD_{N_1,\ell_1,\la_0}^{1/2}(t)\CE_{N_0,\ell_0,\la_0}^{1/2}(t)\}\CD_{N_1,\ell_1,\la_0}^{1/2}(t).
\end{eqnarray*}
Next, for the term $E\cdot \xi f$ in $S$, one has
 \begin{eqnarray*}
&&\lag \pa^\al (E\cdot \xi f),w_{-\ell_1,\la_0}^2\pa^\al f\rag=\sum_{\al_1\leq \al}C_{\al_1}^\al \lag \pa^{\al_1}E\cdot \xi \pa^{\al-\al_1}f,w_{-\ell_1,\la_0}^2\pa^\al f\rag\\
&&\lesssim \sum_{\substack{|\al_1|\leq N_1/2 \\
\al_1\neq \al }}\|\pa^{\al_1}E\|_{L^\infty}\iint_{\R^3\times \R^3}|\xi| w_{-\ell_1,\la_0}^2(|\pa^{\al-\al_1}f|^2+|\pa^\al f |^2)\,dxd\xi \\
&&\quad +\sum_{\substack{|\al_1|>N_1/2 \\
\text{or}\, \al_1=\al }}\|\pa^{\al_1}E\|\cdot \sup_{x}\left||\xi|\lag \xi \rag^{\frac{\ga+2}{2}} w_{-\ell_1,\la_0}\pa^{\al-\al_1}f\right|_{L^2_\xi} \|\lag \xi\rag^{\frac{\ga+2}{2}}w_{-\ell_1,\la_0}\pa^\al f\|\\
&&\lesssim \frac{1}{\la_0} (1+t)^{1+\vth}\|\na_x (E,B)\|_{H^{N_0-1}}  \CD_{N_1,\ell_1,\la_0}(t)\\
&&\quad+\CE_{N_1,\ell_1,\la_0}^{1/2}(t)\CE_{N_0,\ell_0-1,\la_0}^{1/2}(t)\CD_{N_1,\ell_1,\la_0}^{1/2}(t).
\end{eqnarray*}
For the term $(E+\xi\times B)\cdot \na_\xi f$ in $S$, the difference point is that it contains the velocity derivative of order one. One can dedue that
 \begin{eqnarray}
&&\lag \pa^\al [(E+\xi\times B)\cdot \na_\xi f],w_{-\ell_1,\la_0}^2\pa^\al f\rag\notag\\
&&=\lag (E+\xi \times B)\cdot \na_\xi w_{-\ell_1,\la_0}^2, -\frac{1}{2} |\pa^\al f|^2\rag\notag\\
&&\quad +\sum_{0<\al_1\leq \al}C_{\al_1}^\al \lag (\pa^{\al_1}E+\xi\times \pa^{\al_1}B)\cdot \na_\xi \pa^{\al-\al_1}f,w_{-\ell_1,\la_0}^2\pa^\al f\rag.\label{i2.p1}
\end{eqnarray}
Here, it is straightforward to see that the first term on the right is bounded in a rough way by 
\begin{multline}
\notag
C\|E\|_{L^\infty} \iint_{\R^3\times \R^3}(\lag \xi\rag^{-1} +\frac{\lag \xi\rag}{(1+t)^{\vth}})w_{-\ell_1,\la_0}^2|\pa^\al f|^2\,dxd\xi\\
\lesssim \frac{1}{\la_0}(1+t)^{1+\vth}\|\na_x E\|_{H^1}\iint_{\R^3\times \R^3} \frac{\la_0\lag \xi\rag^2}{(1+t)^{1+\vth}}
w_{-\ell_1,\la_0}^2|\pa^\al f|^2\,dxd\xi\\
\lesssim \frac{1}{\la_0}(1+t)^{1+\vth}\|\na_x E\|_{H^1}\CD_{N_1,\ell_1,\la_0}(t).
\end{multline}
The second term on the right-hand side of \eqref{i2.p1} can be estimated as follows.  When $|\al_1|\leq N_1/2$ and $\al_1\neq \al$, it is bounded by
 \begin{eqnarray*}
&&C\|\pa^{\al_1}(E,B)\|_{L^\infty}\iint_{\R^3\times \R^3}\lag \xi\rag^{1-(\ga+2)} (w_{1-\ell_1,\la_0}^2|\na_\xi\pa^{\al-\al_1}f|^2+w_{-\ell_1,\la_0}^2|\pa^\al f |^2)\,dxd\xi \\
&&\lesssim \|\na_x\pa^{\al_1}(E,B)\|_{H^1}\iint_{\R^3\times \R^3}\lag \xi\rag^{2} (|w_{1-\ell_1,\la_0}\na_\xi\pa^{\al-\al_1}f|^2+|w_{-\ell_1,\la_0}\pa^\al f |^2)\,dxd\xi \\
&&\lesssim \frac{1}{\la_0} (1+t)^{1+\vth}\|\na_x (E,B)\|_{H^{N_0-1}}  \CD_{N_1,\ell_1,\la_0}(t),
\end{eqnarray*}
where we have used $-3\leq \ga<-2$, and when $|\al_1|>N_1/2$ or $\al_1=\al$, it is bounded by
 \begin{multline}
 \notag
C\|\pa^{\al_1}(E,B)\|\cdot \sup_{x}\left|\lag \xi \rag^{1+\frac{\ga+2}{2}} w_{-\ell_1,\la_0}\na_\xi\pa^{\al-\al_1}f\right|_{L^2_\xi} \|\lag \xi\rag^{\frac{\ga+2}{2}}w_{-\ell_1,\la_0}\pa^\al f\|\\
\lesssim \CE_{N_1,\ell_1,\la_0}^{1/2}(t)\CE_{N_0,\ell_0-1,\la_0}^{1/2}(t)\CD_{N_1,\ell_1,\la_0}^{1/2}(t).
\end{multline}
Collecting all the above estimates, \eqref{lem.n1w.p5} holds true for the fourth inner product term on the right-hand side of \eqref{def.i2}. This proves \eqref{lem.n1w.p5}.
\end{proof}

Now, we are ready to obtain the closed estimate on the first portion of the time-weighted energy norm $X(t)$ in the following

\begin{lemma}\label{lem.ce1}
It holds that
\begin{equation}
\label{ce1}
\sup_{0\leq s\leq t}\left\{\CE_{N_1}(s) + (1+s)^{-\frac{1+\eps_0}{2}} \CE_{N_1,\ell_1,\la_0} (s)\right\}
+\int_0^t \CD_{N_1}(s)\,ds\lesssim Y_0^2 +X^2(t),
\end{equation}
for $0\leq t\leq T$.
\end{lemma}

\begin{proof}
In fact, the time integration of \eqref{lem.n1.1} gives
\begin{multline}
\CE_{N_1}(t)+\int_0^t\CD_{N_1}(s)\,ds\lesssim Y_0^2+ \de \int_0^t(1+s)^{-1-\vth}\CD_{N_1,\ell_1,\la_0}(s)\,ds\\
+\int_0^t\CE_{N_1}(s)\CE_{N_0,\ell_0-1,\la_0}(s)\,ds.
\label{ce1.p1}
\end{multline}
Furthermore, from multiplying \eqref{lem.n1.1} by $(1+t)^{-\eps_0}$ and then taking the time integration, it follows that
\begin{multline}
\dis (1+t)^{-\eps_0}\CE_{N_1}(t)+\int_0^t(1+s)^{-\eps_0}\CD_{N_1}(s)\,ds +\int_0^t(1+s)^{-1-\eps_0}\CE_{N_1}(s)\,ds\\
\dis \lesssim Y_0^2+\de \int_0^t (1+s)^{-1-\vth-\eps_0}\CD_{N_1,\ell_1,\la_0}(s)\,ds\\
\dis +\int_0^t (1+s)^{-\eps_0}\CE_{N_1}(s)\CE_{N_0,\ell_0-1,\la_0}(s)\,ds.
\label{ce1.p2}
\end{multline}
Combining \eqref{ce1.p1} and  \eqref{ce1.p2} gives
\begin{eqnarray}
&\dis \CE_{N_1}(t)+\int_0^t\CD_{N_1}(s)\,ds +\int_0^t(1+s)^{-1-\eps_0}\CE_{N_1}(s)\,ds\notag\\
&\dis \lesssim Y_0^2+\de \int_0^t (1+s)^{-1-\vth}\CD_{N_1,\ell_1,\la_0}(s)\,ds
+\int_0^t\CE_{N_1}(s)\CE_{N_0,\ell_0-1,\la_0}(s)\,ds\notag\\
&\dis \lesssim Y_0^2+X^2(t)+\de \int_0^t (1+s)^{-1-\vth}\CD_{N_1,\ell_1,\la_0}(s)\,ds,\label{ce1.p3}
\end{eqnarray}
where to obtain the second inequality, we have used
\begin{equation}
\notag
\sup_{0\leq s\leq t} \{\CE_{N_1}(s)+(1+s)^{\frac{3}{2}}\CE_{N_0,\ell_0-1,\la_0}(s)\}\leq X(t).
\end{equation}
From \eqref{lem.n1w.1}, multiplying it by $(1+t)^{-(1+\eps_0)/2}$ and taking the time integration yields
\begin{multline}
\label{ce1.p4}
(1+t)^{-\frac{1+\eps_0}{2}}\CE_{N_1,\ell_1,\la_0}(t)+\int_0^t(1+s)^{-\frac{1+\eps_0}{2}}\CD_{N_1,\ell_1,\la_0}(s)\,ds\\
+\int_0^t(1+s)^{-\frac{3+\eps_0}{2}}\CE_{N_1,\ell_1,\la_0}(s)\,ds \lesssim Y_0^2
+\int_0^t(1+s)^{-\frac{1+\eps_0}{2}}\CE_{N_1,\ell_1,\la_0}(s)\CE_{N_0,\ell_0-1,\la_0}(s)\,ds\\
+ \sum_{|\al|=N_1} \int_0^t (1+s)^{-\frac{1+\eps_0}{2}}\lag \pa^\al E\cdot \xi \mu^{1/2},w_{-\ell_1,\la_0}^2 \pa^\al f\rag.
\end{multline}
The second term on the right is bounded by $X^2(t)$  by noticing again from the definition of $X(t)$
\begin{equation}
\notag
\sup_{0\leq s\leq t} \{(1+s)^{-\frac{1+\eps_0}{2}}\CE_{N_1,\ell_1,\la_0}(s)+(1+s)^{\frac{3}{2}}\CE_{N_0,\ell_0-1,\la_0}(s)\}\leq X(t).
\end{equation}
By the Cauchy-Schwarz inequality, the right-hand third term of \eqref{ce1.p4} is bounded up to a generic constant by
\begin{multline}
\notag
\dis \sum_{|\al|=N_1}\int_0^t (1+s)^{-1-\eps_0} \|\pa^\al E\|^2+\|\lag \xi\rag^{\frac{\ga+2}{2}}\pa^\al f\|^2\,ds\\
\dis \lesssim \int_0^t (1+s)^{-1-\eps_0}\CE_{N_1}(s)\,ds +\int_0^t \CD_{N_1}(s)\,ds.
\end{multline}
Then,  in terms of the above estimates, taking the sum of \eqref{ce1.p3} and \eqref{ce1.p4} and using the fact that $\de>0$ is small enough, we arrive at
\begin{multline}
\label{ce1.p5}
\CE_{N_1}(t)+
(1+t)^{-\frac{1+\eps_0}{2}}\CE_{N_1,\ell_1,\la_0}(t)+\int_0^t\CD_{N_1}(s)\,ds \\
+\int_0^t(1+s)^{-\frac{1+\eps_0}{2}}\CD_{N_1,\ell_1,\la_0}(s)\,ds
+\int_0^t(1+s)^{-1-\eps_0}\CE_{N_1}(s)\,ds\\
+\int_0^t(1+s)^{-\frac{3+\eps_0}{2}}\CE_{N_1,\ell_1,\la_0}(s)\,ds \lesssim Y_0^2+X^2(t).
\end{multline}
Therefore, \eqref{ce1} follows, and then this completes the proof of Lemma \ref{lem.ce1}.
\end{proof}

\subsection{Lyapunov inequality for $\CE_{N_0,\ell_0,\la_0}(t)$} In this section we devote ourselves to obtaining the Lyapunov inequality for $\CE_{N_0,\ell_0,\la_0}(t)$. It is simpler compared to deal with the high-order energy functional $\CE_{N_1,\ell_1,\la_0}(t)$, because the derivatives of the electromagnetic field of order up to $N_0$ decays in time by the definition of $X(t)$.

\begin{lemma}\label{lem.n0w}
For any $\ell$ with $\ell_0-1\leq \ell \leq \ell_0$, there is an energy functional $\CE_{N_0,\ell,\la_0}(t)$ such that 
  \begin{equation}\label{lem.n0w.1}
    \frac{d}{dt}\CE_{N_0,\ell,\la_0}(t)+\CD_{N_0,\ell,\la_0}(t)\lesssim \sum_{|\al|=N_0} \|\pa^\al E\|^2,
  \end{equation}
for $0\leq t\leq T$.

\end{lemma}

\begin{proof}
From the completely same procedure to obtain the energy inequality \eqref{lem.n1w.p4} for $\CE_{N_1,\ell_1,\la_0}(t)$, one can also verify that 
for $\ell$ with $\ell_0-1\leq \ell \leq \ell_0$, there is an energy functional $\CE_{N_0,\ell,\la_0}(t)$ satisfying \eqref{def.e} such that 
\begin{equation}
\label{lem.n0w.p1}
\frac{d}{dt}\CE_{N_0,\ell,\la_0}(t) +\ka \CD_{N_0,\ell,\la_0}(t)\lesssim \CI^{(2)}_{N_0,\ell,\la_0}(t) 
+ \sum_{ |\al|= N_0} \lag\pa^\al E\cdot \xi \mu^{1/2},w_{-\ell,\la_0}^2\pa^\al f \rag.
\end{equation}
where
\begin{eqnarray}
 \CI^{(2)}_{N_0,\ell,\la_0}(t) &=&\lag S, f\rag + \sum_{1\leq |\al|\leq N_0}\lag \pa^\al S,\pa^\al f\rag \notag\\
&& +\lag \{I-P\}S, w_{-\ell,\la_0}^2\{I-P\}f\rag+\sum_{1\leq |\al|\leq N_0} \lag \pa^\al S,w_{-\ell,\la_0}^2\pa^\al f \rag\notag\\
&& +\sum_{m=1}^{N_0}C_m\sum_{\substack{|\be|= m \\
|\al|+|\be|\leq N_0 }}\lag \pa^\al_\be \{I-P\}S,w_{|\be|-\ell,\la_0}^2\pa_\be^\al\{I-P\} f\rag.\label{def.i2n0}
\end{eqnarray}
Compared to the estimate on the similar functional  $ \CI^{(2)}_{N_1,\ell_1,\la_0}(t)$ given by \eqref{def.i2} in Lemma \ref{lem.n1w}, the estimate on $ \CI^{(2)}_{N_0,\ell,\la_0}(t)$ above becomes easier due to the boundedness of $\CE_{N_0,\ell,\la_0}(t)$ uniformly in time and also the time-decay of derivatives of the electromagnetic field up to order $N_0$. In fact, one can claim that
\begin{multline}
\label{lem.n0w.p2}
\CI^{(2)}_{N_0,\ell,\la_0}(t)\lesssim \CE_{N_0,\ell,\la_0}^{1/2}(t) \CD_{N_0,\ell,\la_0}(t) \\
+\frac{1}{\la_0} (1+t)^{1+\vth}\|\na_x (E,B)\|_{H^{N_0-1}}\CD_{N_0,\ell,\la_0}(t).
\end{multline}
Noticing from the definition of $X(t)$
\begin{equation}
\notag
\sup_{0\leq s\leq t}\{\CE_{N_0,\ell_0,\la_0}(s)+ (1+s)^{2(1+\vth)}\|\na_x (E,B)\|_{H^{N_0-1}}^2\}\leq X(t)\leq \de^2,
\end{equation}
it further follows that $\CI^{(2)}_{N_0,\ell,\la_0}(t)$ is bounded up to a generic constant by $ \de\CD_{N_0,\ell,\la_0}(t)$. By putting this estimate into \eqref{lem.n0w.p1} and applying the Cauchy-Schwarz inequality to the right-hand second term of  \eqref{lem.n0w.p1} as
\begin{equation}
\notag
 \sum_{ |\al|= N_0} \lag\pa^\al E\cdot \xi \mu^{1/2},w_{-\ell,\la_0}^2\pa^\al f \rag\leq \eta  \sum_{ |\al|= N_0}\|\lag \xi\rag^{\frac{\ga+2}{2}}\pa^\al f\|^2+\frac{C}{\eta}  \sum_{ |\al|= N_0}\|\pa^\al E\|^2,
\end{equation}
for $\eta>0$ small enough, it follows that
\begin{equation}
\notag
\frac{d}{dt}\CE_{N_0,\ell,\la_0}(t) +\ka \CD_{N_0,\ell,\la_0}(t)\lesssim (\de+\eta)  \CD_{N_0,\ell,\la_0}(t)+\frac{C}{\eta}  \sum_{ |\al|= N_0}\|\pa^\al E\|^2,
\end{equation}
which thus implies \eqref{lem.n0w.1} since $\de>0$ and $\eta>0$ can be small enough. This completes the proof of Lemma \ref{lem.n0w}.
\end{proof}

\begin{proof}[Proof of \eqref{lem.n0w.p2}]
It is similar to the proof of \eqref{lem.n1w.p5} for $\CI^{(2)}_{N_1,\ell_1,\la_0}(t)$. For those inner product terms from $\Ga(f,f)$ in \eqref{def.i2n0}, one can directly apply Lemma \ref{lem.LGa} to verify that they are bounded by $C \CE_{N_0,\ell,\la_0}^{1/2}(t) \CD_{N_0,\ell,\la_0}(t)$. For those terms related to $(E,B)$ in \eqref{def.i2n0}, from the completely same process as for dealing  in the proof of \eqref{lem.n1w.p5}  with the case when $(E,B)$ gains the differentiation of lower-order, it follows that they are bounded by $C\frac{1}{\la_0} (1+t)^{1+\vth}\|\na_x (E,B)\|_{H^{N_0-1}}\CD_{N_0,\ell,\la_0}(t)$. Therefore,  \eqref{lem.n0w.p2} is proved.
\end{proof}

\subsection{Decay of electromagnetic fields and macro components}

In this step, we will use directly the Duhamel's principle to obtain the time-decay  of the electromagnetic field $(E,B)$ and the macro components $(a,b,c)$ up to the low-order $N_0$ in terms of the time-decay of the weighted high-order energy function $\CE_{N_1-3,\ell_1-1,\la_0}(t)$ which follows from the boundedness of $X(t)$.

\begin{lemma}\label{lem.emd}
It holds that
\begin{equation}
\label{lem.emd.1}
\sup_{0\leq s\leq t}\{(1+s)^{\frac{5}{2}}\|\na_x (E,B)\|_{H^{N_0-1}}^2+(1+s)^{\frac{3}{2}}\|(a,b,c,E,B)\|^2\}\lesssim Y_0^2+X^2(t),
\end{equation}
for $0\leq t\leq T$.
\end{lemma}

\begin{proof}
Recall the mild form
\begin{equation}
U(t)=\semiG(t)U_0+\int_0^t\semiG(t-s)[\sourceG(s),0,0]\,ds,\label{lem.emd.p1}
\end{equation}
which denotes the solutions to  the Cauchy problem on the Vlasov-Maxwell-Landau system \eqref{ns} with initial data $U_0=(f_0,E_0,B_0)$, where the nonlinear term $S$ is given by \eqref{def.S}. The linearized analysis for the homogeneous system in Theorem \ref{thm.lb} implies
\begin{multline}
\notag
 \|\na_x P_{E,B}\{\semiG(t)U_0\}\|_{H^{N_0-1}} \lesssim (1+t)^{-\frac{5}{4}} (\|w^{\ell_3^{low}} f_0\|_{Z_1}+\|(E_0,B_0)\|_{L^1_x})\\
 + (1+t)^{-\frac{5}{4}} \sum_{1\leq |\al|\leq N_0} (\|w^{\ell_3^{high}} \na_x^{\frac{9}{4}}\pa^\al f_0\|+\|\na_x^{\frac{9}{4}} \pa^\al (E_0,B_0)\|),
\end{multline}
where $P_{E,B}$ means the projection along the electro and magnetic components  in the solution $(f,E,B)$, $w=w(\xi)$ is defined by \eqref{def.wli}, and  constants $\ell_3^{low}$, $\ell_3^{high}$ are chosen to satisfy
\begin{equation}
\notag
\ell_3^{low}>\frac{5}{2} +N_0,\quad \ell_3^{high}>\frac{5}{4},
\end{equation}
and also $\ell_3^{low}$, $\ell_3^{high}$ are sufficiently close to $3/2+N_0$ and $5/4$, respectively.
By interpolation of derivatives,
\begin{multline}
\notag
 \|\na_x P_{E,B}\{\semiG(t)U_0\}\|_{H^{N_0-1}} \lesssim (1+t)^{-\frac{5}{4}} (\|w^{\ell_3^{low}} f_0\|_{Z_1}+\|(E_0,B_0)\|_{L^1_x})\\
+ (1+t)^{-\frac{5}{4}} \sum_{3\leq |\al|\leq N_0+3} (\|w^{\ell_3^{high}} \pa^\al f_0\|+\| \pa^\al (E_0,B_0)\|).
\end{multline}
Applying this time-decay property to the mild form \eqref{lem.emd.p1} gives
\begin{multline}
\label{lem.emd.p2}
 \|\na_x (E,B)\|_{H^{N_0-1}}
\lesssim (1+t)^{-\frac{5}{4}}Y_0
+\int_0^t(1+t-s)^{-\frac{5}{4}} \|w^{\ell_3^{low}} S(s)\|_{Z_1}\,ds\\
 + \int_0^t(1+t-s)^{-\frac{5}{4}} \sum_{3\leq |\al|\leq N_0+3} \|w^{\ell_3^{high}} \pa^\al S(s)\|\,ds,
\end{multline}
where we have used $w(\xi)=w_{-\frac{1}{2}} (\xi)$ and the definition \eqref{def.Y0} for $Y_0$.
As in \cite{DYZ-VPL}, it is straightforward to obtain
\begin{equation}
\notag
 \|w^{\ell_3^{low}} S(t)\|_{Z_1}+\sum_{3\leq |\al|\leq N_0+3} \|w^{\ell_3^{high}} \pa^\al S(t)\|\lesssim \CE_{N_1-3,\ell_1-1,\la_0}(t).
\end{equation}
Here, we have used the choice of $N_1$, $\ell_1$ by $N_1=\frac{3}{2}N_0$, $\ell_1=\frac{1}{2}\ell_0$ with $N_0$ and $\ell_0$ properly large.
Recall $X(t)$ norm, and hence
\begin{equation}
\notag
 \CE_{N_1-3,\ell_1-1,\la_0}(s)\leq (1+s)^{-\frac{3}{2}}X(t),\quad 0\leq s\leq t.
\end{equation}
Plugging these estimates into \eqref{lem.emd.p2}, the further computations yield
\begin{equation}
\label{lem.emd.p3}
 \sup_{0\leq s\leq t} \{(1+s)^{\frac{5}{2}} \|\na_x (E,B)\|_{H^{N_0-1}}^2\}\lesssim Y_0^2 +X^2(t).
\end{equation}
Moreover, to obtain the time-decay of $\|(a,b,c,E,B)\|$, we use the linearized time-decay property
\begin{multline}
\notag
\|P_f\{\semiG(t)U_0\}\|+\|P_{E,B}\{\semiG(t)U_0\}\| \lesssim (1+t)^{-\frac{3}{4}} (\|w^{\ell_4^{low}} f_0\|_{Z_1}+\|(E_0,B_0)\|_{L^1_x})\\
 + (1+t)^{-\frac{3}{4}} (\|w^{\ell_4^{high}} \na_x^{\frac{7}{4}} f_0\|+\|\na_x^{\frac{7}{4}}  (E_0,B_0)\|),
\end{multline}
where $P_{f}$ means the projection along the  $f$-component in the solution $(f,E,B)$, and  constants $\ell_4^{low}$, $\ell_4^{high}$ are chosen to satisfy
$\ell_4^{low}>3/2$, $\ell_4^{high}>3/4$
and also $\ell_4^{low}$, $\ell_4^{high}$ are sufficiently close to $3/2$ and $3/4$, respectively. Therefore, in the completely same way for estimating $\|\na_x (E,B)\|_{H^{N_0-1}}$ in \eqref{lem.emd.p3}, one has
\begin{equation}
\label{lem.emd.p4}
 \sup_{0\leq s\leq t} \{(1+s)^{\frac{3}{2}} \|(a,b,c,E,B)\|^2\}\lesssim Y_0^2 +X^2(t).
\end{equation}
Thus, combining \eqref{lem.emd.p3} and \eqref{lem.emd.p4} gives the desired estimate \eqref{lem.emd.1}. This then completes the proof of Lemma \ref{lem.emd}.
\end{proof}

\subsection{Bound of $\CE_{N_0,\ell_0,\la_0}(t)$ and decay of $\CE_{N_0,\ell_0-1,\la_0}(t)$} 
Basing on those estimates in the previous two sections, we can use the time-weighted estimate together with an iterative trick to obtain in the following lemma the boudedness of $\CE_{N_0,\ell_0,\la_0}(t)$ and also the time-decay of $\CE_{N_0,\ell_0-1,\la_0}(t)$. Notice that loss of one in the velocity weight index of $\CE_{N_0,\ell_0-1,\la_0}(t)$ results essentially from the long-range degenerate property of soft potentials for the Landau operator.

\begin{lemma}\label{lem.n0wd}
It holds that
\begin{equation}
\label{lem.n0wd.1}
\sup_{0\leq s\leq t} \{\CE_{N_0,\ell_0,\la_0}(s)+(1+s)^{\frac{3}{2}}\CE_{N_0,\ell_0-1,\la_0}(s)\} \lesssim Y_0^2 +X^2(t),
\end{equation}
for $0\leq t\leq T$.
\end{lemma}

\begin{proof}
Recall the Lyapunov inequality \eqref{lem.n0w.1} for $\CE_{N_0,\ell,\la_0}(t)$ with $\ell_0-1\leq \ell\leq \ell_0$. First of all, from the time integration of \eqref{lem.n0w.1} with $\ell=\ell_0$,
\begin{equation}\label{lem.n0wd.p1}
\CE_{N_0,\ell_0,\la_0}(t)+\int_0^t\CD_{N_0,\ell_0,\la_0}(s)\,ds \lesssim Y_0^2 + \sum_{|\al|=N_0}\int_0^t \|\pa^\al E\|^2\,ds.
\end{equation}
Due to Lemma \ref{lem.emd}, the second term on the right is bounded up to a generic constant by $Y_0^2+X^2(t)$, and so is $\CE_{N_0,\ell_0,\la_0}(s)$ for $0\leq s\leq t$. The rest is to estimate the time-weighted part on the left-hand side of \eqref{lem.n0wd.1}. For that, let $\eps>0$ be fixed small enough. From multiplying \eqref{lem.n0w.1} with $\ell=\ell_0-1/2$ by $(1+t)^{\frac{1}{2}+\eps}$ and then taking the time integration,
\begin{multline}
\label{lem.n0wd.p2}
 (1+t)^{\frac{1}{2}+\eps}\CE_{N_0,\ell_0-\frac{1}{2},\la_0}(t)+\int_0^t(1+s)^{\frac{1}{2}+\eps} \CD_{N_0,\ell_0-\frac{1}{2},\la_0}(s)\,ds
 \lesssim Y_0^2\\+\int_0^t(1+s)^{-\frac{1}{2}+\eps} \CE_{N_0,\ell_0-\frac{1}{2},\la_0}(s)\,ds
 + \sum_{|\al|=N_0}\int_0^t(1+s)^{\frac{1}{2}+\eps}\|\pa^\al E\|^2\,ds,
\end{multline}
and in a similar way, starting from \eqref{lem.n0w.1} with $\ell=\ell_0-1$,
\begin{multline}
\label{lem.n0wd.p3}
 (1+t)^{\frac{3}{2}+\eps}\CE_{N_0,\ell_0-1,\la_0}(t) +\int_0^t(1+s)^{\frac{3}{2}+\eps} \CD_{N_0,\ell_0-1,\la_0}(s)\,ds
 \lesssim Y_0^2\\+\int_0^t(1+s)^{\frac{1}{2}+\eps} \CE_{N_0,\ell_0-1,\la_0}(s)\,ds
 + \sum_{|\al|=N_0}\int_0^t(1+s)^{\frac{3}{2}+\eps}\|\pa^\al E\|^2\,ds.
\end{multline}
Using the relation between the energy functional $\CE_{N,\ell,\la}(t)$ and its dissipation rate $\CD_{N,\ell,\la}(t)$, the proper linear combination of \eqref{lem.n0wd.p1}, \eqref{lem.n0wd.p2} and \eqref{lem.n0wd.p3} yields
\begin{multline}
\label{lem.n0wd.p4}
(1+t)^{\frac{3}{2}+\eps}\CE_{N_0,\ell_0-1,\la_0}(t)
+\int_0^t\CD_{N_0,\ell_0,\la_0}(s)+(1+s)^{\frac{3}{2}+\eps} \CD_{N_0,\ell_0-1,\la_0}(s)\,ds\\
 \lesssim Y_0^2
 +\int_0^t(1+s)^{\frac{1}{2}+\eps} \|(a,b,c,B)\|^2\,ds\\
 + \sum_{|\al|=N_0}\int_0^t(1+s)^{\frac{3}{2}+\eps}\|\pa^\al (E,B)\|^2\,ds.
\end{multline}
Again from Lemma \ref{lem.emd}, one has
\begin{multline}
\notag
\int_0^t(1+s)^{\frac{1}{2}+\eps} \|(a,b,c,B)\|^2\,ds
 + \sum_{|\al|=N_0}\int_0^t(1+s)^{\frac{3}{2}+\eps}\|\pa^\al (E,B)\|^2\,ds\\
 \lesssim \int_0^t (1+s)^{-1+\eps}\,ds\,[Y_0^2+X^2(t)] \lesssim (1+t)^{\eps}[Y_0^2+X^2(t)].
\end{multline}
Using this, it follows from \eqref{lem.n0wd.p4} that
\begin{equation}
\notag
\sup_{0\leq s\leq t}\{(1+s)^{\frac{3}{2}}\CE_{N_0,\ell_0-1,\la_0}(s)\}\lesssim Y_0^2+X^2(t).
\end{equation}
Therefore \eqref{lem.n0wd.1} holds true. This completes the proof of Lemma \ref{lem.n0wd}.
\end{proof}

\subsection{Bound of $\CE_{N_1-1,\ell_1,\la_0}(t)$}

In this section we obtain the uniform-in-time boundedness of the energy functional $\CE_{N_1-1,\ell_1,\la_0}(t)$. Notice that this is consistent with \eqref{ad.p1} in the linearized analysis. The main observation in the nonlinear analysis is that those remaining terms in the energy inequalities are time-space integrable.

\begin{lemma}\label{lem.n1bd}
It holds that
\begin{equation}
\label{lem.n1bd.1}
\sup_{0\leq s\leq t} \CE_{N_1-1,\ell_1,\la_0}(s)+\int_0^t \CD_{N_1-1,\ell_1,\la_0}(s)\,ds\lesssim Y_0^2+X^2(t),
\end{equation}
for $0\leq t\leq T$.
\end{lemma}

\begin{proof}
Similarly for obtaining \eqref{lem.n1w.p4}, one has
\begin{multline}
\label{lem.n1bd.p1}
\frac{d}{dt}\CE_{N_1-1,\ell_1,\la_0}(t) +\ka \CD_{N_1-1,\ell_1,\la_0}(t)\lesssim \CI^{(2)}_{N_1-1,\ell_1,\la_0}(t) \\
+ \sum_{ |\al|= N_1-1} \lag\pa^\al E\cdot \xi \mu^{1/2},w_{-\ell_1,\la_0}^2\pa^\al f \rag,
\end{multline}
where $ \CI^{(2)}_{N_1-1,\ell_1,\la_0}(t)$ is defined by \eqref{def.i2} with $N_1$ replaced by $N_1-1$. 
The first term on the right can be bounded by
\begin{multline}
\label{lem.n1bd.p1-0}
 \CI^{(2)}_{N_1-1,\ell_1,\la_0}(t)\lesssim \CE_{N_1-1,\ell_1,\la_0}^{1/2}(t)\CD_{N_1-1,\ell_1,\la_0}(t) \\
 + \frac{1}{\la_0} (1+t)^{1+\vth} \|\na_x (E,B)\|_{H^{N_0-1}}\CD_{N_1-1,\ell_1,\la_0}(t)\\
 +\CE_{N_1-1}^{1/2}(t)\CE_{N_0,\ell_0-1,\la_0}^{1/2}(t) \CD_{N_1-1,\ell_1,\la_0}^{1/2}(t).
\end{multline}
Further using 
\begin{multline}
\notag
\sup_{0\leq s\leq t}\{\CE_{N_1-1}(s)+\CE_{N_1-1,\ell_1,\la_0}(s)+(1+s)^{2(1+\vth)} \|\na_x (E,B)\|_{H^{N_0-1}}^2\\
+(1+s)^{\frac{3}{2}}\CE_{N_0,\ell_0-1,\la_0}(s) \}\leq X(t)\leq \de^2,
\end{multline}
it follows that
\begin{equation}
\label{lem.n1bd.p2}
 \CI^{(2)}_{N_1-1,\ell_1,\la_0}(t)\lesssim \de\CD_{N_1-1,\ell_1,\la_0}(t)+(1+t)^{-\frac{3}{4}} X(t)\CD_{N_1-1,\ell_1,\la_0}^{1/2}(t).
\end{equation}
From the Cauchy-Schwarz inequality, the right-hand second term of \eqref{lem.n1bd.p1} is estimated by
\begin{multline}
\label{lem.n1bd.p3}
 \sum_{ |\al|= N_1-1} \lag\pa^\al E\cdot \xi \mu^{1/2},w_{-\ell_1,\la_0}^2\pa^\al f \rag\lesssim  \sum_{ |\al|= N_1-1} (\eta\|\lag \xi\rag^{\frac{\ga+2}{2}}\pa^\al f\|^2+\frac{1}{\eta}\|\pa^\al E\|^2\|)\\
 \lesssim \eta \CD_{N_1-1,\ell_1,\la_0}(t)+\frac{1}{\eta}\CD_{N_1}(t),
\end{multline}
for any $\eta>0$. Then, by  applying again the Cauchy-Schwarz inequality with $\eta$ to the right-hand second term of \eqref{lem.n1bd.p2}, plugging the resultant estimate together with \eqref{lem.n1bd.p3} into \eqref{lem.n1bd.p1}, and choosing $\eta>0$ small enough, one has\begin{equation}
\label{lem.n1bd.p4}
\frac{d}{dt}\CE_{N_1-1,\ell_1,\la_0}(t) +\ka \CD_{N_1-1,\ell_1,\la_0}(t)\lesssim \CD_{N_1}(t) + (1+t)^{-\frac{3}{2}}X^2(t).
\end{equation}
Recall that from \eqref{ce1.p5},
\begin{equation}
\notag
\int_0^t\CD_{N_1}(s)\,ds\lesssim Y_0^2+X^2(t). 
\end{equation}
Therefore, \eqref{lem.n1bd.1} follows by the time integration of \eqref{lem.n1bd.p4}. This completes the proof of Lemma \ref{lem.n1bd}.
\end{proof}

\subsection{Decay of $\CE_{N_1-3,\ell_1-1,\la_0}(t)$ and $\CE_{N_1-2}(t)$}

To obtain the closed estimate on the energy norm $X(t)$, it remains to obtain the time-decay of the high-order energy functional  $\CE_{N_1-3,\ell_1-1,\la_0}(t)$ and $\CE_{N_1-2}(t)$ through the time-weighted estimate as well as the iterative trick as for dealing with $\CE_{N_0,\ell_0-1,\la_0}(t)$ in Lemma \ref{lem.n0wd}. Notice that loss of three and loss of one in the smoothness and velocity weight indices  in $\CE_{N_1-3,\ell_1-1,\la_0}(t)$ respectively result from the regularity-loss of the electromagnetic field and the degeneration of collisional kernels for soft potentials.

\begin{lemma}\label{lem.n1wd}
It holds that
\begin{equation}
\label{lem.n1wd.1}
\sup_{0\leq s\leq t}\{(1+s)^{\frac{3}{2}}[\CE_{N_1-3,\ell_1-1,\la_0}(s)+\CE_{N_1-2}(s)]\}\lesssim Y_0^2+X^2(t),
\end{equation}
for $0\leq t\leq T$.
\end{lemma}

\begin{proof}
First recall from Lemma \ref{lem.ce1} and Lemma \ref{lem.n1bd}
\begin{equation}
\label{lem.n1wd.p1}
\CE_{N_1}(t)+\CE_{N_1-1,\ell_1,\la_0}(t)+ \int_0^t \CD_{N_1}(s)+\CD_{N_1-1,\ell_1,\la_0}(s)\,ds \lesssim Y_0^2+X^2(t).
\end{equation}
To obtain the time-decay of $\CE_{N_1-3,\ell_1-1,\la_0}(t)$ and $\CE_{N_1-2}(t)$, we will make the time-weighted estimate. For brevity of presentation we write
\begin{equation}
\notag
\CJ^{(2)}_{N,\ell,\la_0}(t)=\sum_{|\al|=N}\lag \pa^\al E\cdot \xi \mu^{1/2}, w_{-\ell,\la_0}^2\pa^\al f\rag.
\end{equation}
From the proof of Lemma \ref{lem.n1} and Lemma \ref{lem.n1w}, cf.~\eqref{lem.n1.p5} and \eqref{lem.n1w.p4}, one has the Lyapunov inequalities
\begin{equation}
\label{lem.n1wd.p2}
\left\{\begin{split}
&\frac{d}{dt} \CE_{N_1-1}(t)+\ka \CD_{N_1-1}(t)\lesssim \CI^{(1)}_{N_1-1}(t),\\
&\dis \frac{d}{dt} \CE_{N_1-2,\ell_1-\frac{1}{2},\la_0}(t)+\kappa \CD_{N_1-2,\ell_1-\frac{1}{2},\la_0}(t)\\
&\qquad\qquad\qquad\qquad\qquad\qquad\lesssim \CI^{(2)}_{N_1-2,\ell_1-\frac{1}{2},\la_0}(t) + \CJ^{(2)}_{N_1-2,\ell_1-\frac{1}{2},\la_0}(t).
\end{split}\right.
\end{equation}
Those terms on the right can be estimated as follows. Similar to \eqref{lem.n1.p6}, it holds that
\begin{multline}
\notag
\CI^{(1)}_{N_1-1}(t)\lesssim  \CE_{N_1-1}^{1/2}(t)\CD_{N_1-1}(t) +\frac{\de}{(1+t)^{1+\vth}}\CD_{N_1-1,\ell_1,\la_0}(t)\\
+\CE_{N_1-1}^{1/2}(t) \CE_{N_0,\ell_0-1,\la_0}^{1/2}(t)\CD_{N_1-1}^{1/2}(t).
\end{multline}
Here, noticing that $ \CE_{N_1-1}^{1/2}(t)\leq X^{1/2}(t)\leq \de$ is small enough for the first term on the right and applying the Cauchy-Schwarz inequality  to the third term on the right, it then follows from the first equation of \eqref{lem.n1wd.p2} that
\begin{multline}
\label{lem.n1wd.p3}
\frac{d}{dt} \CE_{N_1-1}(t)+\ka \CD_{N_1-1}(t)\\
\lesssim \frac{\de}{(1+t)^{1+\vth}}\CD_{N_1-1,\ell_1,\la_0}(t)+\CE_{N_1-1}(t) \CE_{N_0,\ell_0-1,\la_0}(t).
\end{multline}
Moreover, similar to \eqref{lem.n1bd.p1-0}, it holds that
\begin{multline}
\notag
 \CI^{(2)}_{N_1-2,\ell_1-\frac{1}{2},\la_0}(t)\lesssim \CE_{N_1-2,\ell_1-\frac{1}{2},\la_0}^{1/2}(t)\CD_{N_1-2,\ell_1-\frac{1}{2},\la_0}(t) \\
 + \frac{1}{\la_0} (1+t)^{1+\vth} \|\na_x (E,B)\|_{H^{N_0-1}}\CD_{N_1-2,\ell_1-\frac{1}{2},\la_0}(t)\\
 +\CE_{N_1-2}^{1/2}(t)\CE_{N_0,\ell_0-1,\la_0}^{1/2}(t) \CD_{N_1-2,\ell_1-\frac{1}{2},\la_0}^{1/2}(t),
\end{multline}
which by using $X(t)\leq \de^2$ for the first two terms on the right and the Cauchy-Schwarz inequality for the last term, further implies
\begin{equation}
\label{lem.n1wd.p4}
 \CI^{(2)}_{N_1-2,\ell_1-\frac{1}{2},\la_0}(t)\lesssim (\de+\eta)\CD_{N_1-2,\ell_1-\frac{1}{2},\la_0}(t) 
 +\frac{1}{\eta}\CE_{N_1-2}(t)\CE_{N_0,\ell_0-1,\la_0}(t),
 \end{equation}
for $\eta>0$. Again from the Cauchy-Schwarz inequality with $\eta>0$,
\begin{equation}
\label{lem.n1wd.p5}
 \CJ^{(2)}_{N_1-2,\ell_1-\frac{1}{2},\la_0}(t)\lesssim \sum_{|\al|=N_1-2}(\eta\|\lag \xi\rag^{\frac{\ga+2}{2}}\pa^\al f\|^2+ \frac{1}{\eta}\|\pa^\al E\|^2).
\end{equation}
Then, by plugging \eqref{lem.n1wd.p4} and \eqref{lem.n1wd.p5} into the second equation of \eqref{lem.n1wd.p2}, taking the sum of the resultant inequality multiplied by a proper small constant {$\kappa_3>0$} and another inequality \eqref{lem.n1wd.p3}, and using smallness of $\de>0$ and $\eta>0$, one has
\begin{multline}
\label{lem.n1wd.p6}
\frac{d}{dt} \{\CE_{N_1-1}(t)+\kappa_3 \CE_{N_1-2,\ell_1-\frac{1}{2},\la_0}(t)\}
+\kappa \{\CD_{N_1-1}(t)+\kappa_3 \CD_{N_1-2,\ell_1-\frac{1}{2},\la_0}(t)\}\\
\lesssim \frac{\de}{(1+t)^{1+\vth}} \CD_{N_1-1,\ell_1,\la_0}(t) +\CE_{N_1-1}(t)\CE_{N_0,\ell_0-1,\la_0}(t).
\end{multline}
Further from multiplying it by $(1+t)^{\frac{1}{2}+\eps}$ with $\eps>0$ fixed small enough and taking the time integration, it follows 
\begin{multline}
\label{lem.n1wd.p7}
(1+t)^{\frac{1}{2}+\eps}\{\CE_{N_1-1}(t)+\CE_{N_1-2,\ell_1-\frac{1}{2},\la_0}(t)\}\\
+\int_0^t (1+s)^{\frac{1}{2}+\eps}\{\CD_{N_1-1}(s)+\CD_{N_1-2,\ell_1-\frac{1}{2},\la_0}(s)\}\,ds\\
\lesssim Y_0^2 
+\int_0^t \de(1+s)^{-\frac{1}{2}-\vth+\eps} \CD_{N_1-1,\ell_1,\la_0}(s)\,ds \\
+\int_0^t  (1+s)^{\frac{1}{2}+\eps}\CE_{N_1-1}(s)\CE_{N_0,\ell_0-1,\la_0}(s)\,ds\\
+\int_0^t(1+s)^{-\frac{1}{2}+\eps}\{\CE_{N_1-1}(s)+\CE_{N_1-2,\ell_1-\frac{1}{2},\la_0}(s)\}\,ds.
\end{multline}
Here, since $\eps>0$ is small enough, the second term on the right is bounded by $Y_0^2+X^2(t)$ directly by \eqref{lem.n1wd.p1}, the third term on the right is bounded by 
\begin{equation}
\notag
\de^2\sup_{0\leq s\leq t}\{(1+s)^{\frac{1}{2}+\eps}\CE_{N_1-1}(s)\},
\end{equation}
due to the fact that 
\begin{equation}
\notag
\sup_{0\leq s\leq t}\{(1+s)^{\frac{3}{2}}\CE_{N_0,\ell_0-1,\la_0}(s)\}\leq X(t)\leq \de^2,
\end{equation}
and the fourth term on the right is bounded by $Y_0^2+X^2(t)$ by noticing
\begin{equation}
\notag
 \CE_{N_1-1}(t)+ \CE_{N_1-2,\ell_1-\frac{1}{2},\la_0}(t)\lesssim \CD_{N_1}(t)+\CD_{N_1-1,\ell_1,\la_0}(t)+\|(a,b,c,B)\|^2,
\end{equation}
and further using \eqref{lem.n1wd.p1} as well as Lemma \ref{lem.emd}. Hence, we arrive from \eqref{lem.n1wd.p7} at
\begin{multline}
\label{lem.n1wd.p8}
\sup_{0\leq s\leq t}\{(1+s)^{\frac{1}{2}+\eps}[\CE_{N_1-1}(s)+\CE_{N_1-2,\ell_1-\frac{1}{2},\la_0}(s)]\}\\
+\int_0^t (1+s)^{\frac{1}{2}+\eps}\{\CD_{N_1-1}(s)+\CD_{N_1-2,\ell_1-\frac{1}{2},\la_0}(s)\}\,ds
\lesssim Y_0^2 +X^2(t).
\end{multline}
In a similar way to obtain \eqref{lem.n1wd.p6}, starting with the Lyapunov inequalities
\begin{equation}
\notag
\left\{\begin{split}
&\frac{d}{dt} \CE_{N_1-2}(t)+\ka \CD_{N_1-2}(t)\lesssim \CI^{(1)}_{N_1-2}(t),\\
&\dis \frac{d}{dt} \CE_{N_1-3,\ell_1-1,\la_0}(t)+\kappa \CD_{N_1-3,\ell_1-1,\la_0}(t)\\
&\qquad\qquad\qquad\qquad\qquad\qquad\lesssim \CI^{(2)}_{N_1-3,\ell_1-1,\la_0}(t) + \CJ^{(2)}_{N_1-3,\ell_1-1,\la_0}(t),
\end{split}\right.
\end{equation}
one can prove
\begin{multline}
\notag
\frac{d}{dt} \{\CE_{N_1-2}(t)+{\ka_{4}} \CE_{N_1-3,\ell_1-1,\la_0}(t)\}
+\kappa \{\CD_{N_1-2}(t)+\kappa_{4} \CD_{N_1-2,\ell_1-1,\la_0}(t)\}\\
\lesssim \frac{\de}{(1+t)^{1+\vth}} \CD_{N_1-2,\ell_1-\frac{1}{2},\la_0}(t) +\CE_{N_1-2}(t)\CE_{N_0,\ell_0-1,\la_0}(t),
\end{multline}
for a properly chosen constant $\ka_{4}>0$. Further multiplying it by $(1+t)^{\frac{3}{2}+\eps}$ and taking the time integration gives 
\begin{multline}
\label{lem.n1wd.p9}
(1+t)^{\frac{3}{2}+\eps}\{\CE_{N_1-2}(t)+\CE_{N_1-3,\ell_1-1,\la_0}(t)\}\\
+\int_0^t (1+s)^{\frac{3}{2}+\eps}\{\CD_{N_1-2}(s)+\CD_{N_1-3,\ell_1-1,\la_0}(s)\}\,ds\\
\lesssim Y_0^2 
+\int_0^t \de(1+s)^{\frac{1}{2}+\eps-\vth} \CD_{N_1-2,\ell_1-\frac{1}{2},\la_0}(s)\,ds \\
+\int_0^t  (1+s)^{\frac{3}{2}+\eps}\CE_{N_1-2}(s)\CE_{N_0,\ell_0-1,\la_0}(s)\,ds\\
+\int_0^t(1+s)^{\frac{1}{2}+\eps}\{\CE_{N_1-2}(s)+\CE_{N_1-3,\ell_1-1,\la_0}(s)\}\,ds.
\end{multline}
Here, notice again that $\eps>0$ is a fixed constant small enough. Then, the second term on the right is bounded by $Y_0^2+X^2(t)$  by \eqref{lem.n1wd.p8}, the third term on the right is bounded by  $X^2(t)$ due to 
\begin{equation}
\notag
\sup_{0\leq s\leq t}\{(1+s)^{\frac{3}{2}}[\CE_{N_1-2}(s)+\CE_{N_0,\ell_0-1,\la_0}(s)]\}\leq X(t),
\end{equation}
and as before, the fourth term on the right is bounded by 
\begin{equation}
\notag
C(1+t)^{\eps}[Y_0^2+X^2(t)]
\end{equation}
by noticing
\begin{equation}
\notag
 \CE_{N_1-2}(t)+ \CE_{N_1-3,\ell_1-1,\la_0}(t)\lesssim \CD_{N_1-1}(t)+\CD_{N_1-2,\ell_1-\frac{1}{2},\la_0}(t)+\|(a,b,c,B)\|^2,
\end{equation}
and further using \eqref{lem.n1wd.p8} as well as Lemma \ref{lem.emd}. Therefore,  the desired inequality \eqref{lem.n1wd.1} follows by putting these estimates into \eqref{lem.n1wd.p9}. This then completes the proof of Lemma \ref{lem.n1wd}.
\end{proof}

\subsection{Global existence}

We are now in a position to complete the 

\begin{proof}[Proof of Theorem \ref{thm.gl}]

Recall $X$-norm. From Lemma \ref{lem.ce1}, Lemma \ref{lem.emd}, Lemma \ref{lem.n0wd}, Lemma \ref{lem.n1bd} and Lemma \ref{lem.n1wd}, it follows that
\begin{equation}
\notag
X(t)\lesssim Y_0^2+X^2(t).
\end{equation}
Since $Y_0$ is sufficiently small, \eqref{thm.gl.1} holds true. 
The global existence follows.
\end{proof}

\noindent {\bf Acknowledgements:} This work was supported by the General
Research Fund (Project No.~400511) from RGC of Hong Kong. The author would thank Shuangqian Liu, Tong Yang and Huijiang Zhao for their fruitful discussions on the topic, and also Robert Strain for pointing out his work \cite{St-Op}.

\end{document}